\documentclass{amsart}

\usepackage{amsmath,amssymb,amsthm,mathtools,color,url}
\usepackage[margin=3cm]{geometry}
\usepackage{tikz}

\newtheorem{lemma}{Lemma}[section]
\newtheorem{prop}[lemma]{Proposition}
\newtheorem{cor}[lemma]{Corollary}
\newtheorem{thm}[lemma]{Theorem}
\newtheorem{thm?}[lemma]{Theorem?}

\theoremstyle{definition}
\newtheorem{example}[lemma]{Example}
\theoremstyle{remark}
\newtheorem{remark}[lemma]{Remark}

\newcommand{\F}{\mathbb{F}}

\newcommand{\ra}{\ensuremath{\rightarrow}}

\newcommand{\Z}{\mathbb{Z}}
\newcommand{\N}{\mathbb{N}}
\newcommand{\R}{\mathbb{R}}

\newcommand{\Q}{\mathbb{Q}}

\newcommand{\SL}{\operatorname{SL}}

\newcommand{\disc}{\operatorname{disc}}

\newcommand{\loc}{\operatorname{loc}}
\newcommand{\legen}[2]{\genfrac{(}{)}{}{}{#1}{#2}}
\renewcommand{\P}{\mathbb{P}}

\begin{document}
\title{Densities of integer sets represented by quadratic forms}
\author{Pete L. Clark}
\author{Paul Pollack}

\author{Jeremy Rouse}

\author{Katherine Thompson}

\maketitle

\begin{abstract}
Let $f(t_1,\ldots,t_n)$ be a nondegenerate integral quadratic form.  We analyze the asymptotic behavior of the function $D_f(X)$, the number of integers of absolute value up to $X$ represented by $f$.  When $f$ is isotropic or $n$ is at least $3$, we show that there is a $\delta(f) \in \Q \cap (0,1)$ such that $D_f(X) \sim \delta(f) X$ and call $\delta(f)$ the \textbf{density} of $f$.  We consider the inverse problem of which 
densities arise.  Our 
main technical tool is a Near Hasse Principle: a quadratic form may fail to represent infinitely many integers that it locally 
represents, but this set of exceptions has density $0$ within the set of locally represented integers.
\end{abstract}


\section{Introduction}

\subsection{Terminological preliminaries}
Let $S$ be a subset of $\Z$.  We say that $S$ is \textbf{positive} if it contains no negative integers,
\textbf{negative} if it contains no positive integers, \textbf{definite} if it is either positive or negative and \textbf{indefinite} otherwise.
\\ \\
We say a subset $S \subset \Z$ has \textbf{density} $\delta = \delta(S)$ if
\[ \begin{cases}  \lim_{N \ra \infty} \frac{\# S \cap [1,N]}{N}  = \delta & \text{ if }S \text{ is positive}, \\
\lim_{N \ra \infty} \frac{\# S \cap [-N,-1]}{N} = \delta & \text{ if }S \text{ is negative}, \\
\lim_{N \ra \infty} \frac{\# S \cap [-N,N]}{2N} = \delta & \text{ if }S \text{ is indefinite};
\end{cases} \]
in each case we are requiring the limit to exist.  We say that $S \subset \Z$ has \textbf{upper density} $\overline{\delta}$ (
resp. \textbf{lower density} $\underline{\delta}$) if in the above definition we replace $\lim_{N \ra \infty}$ by $\limsup_{N \ra \infty}$ (resp. by $\liminf_{N \ra \infty}$). 
\\ \\
Let $R$ be a PID of characteristic different from $2$ with fraction field $K$.  We denote by $R^{\bullet}$ the set of nonzero elements of $R$.  For a prime element $p$ of $R$, let
$v_p$ be the corresponding $p$-adic valuation. Let $f \in R[t_1,\ldots,t_n]$ be a nondegenerate quadratic form.  (Henceforth we assume all quadratic forms to be nondegenerate.)  We put
\[ D_f \coloneqq f(R^n) \setminus \{0\}. \]
A quadratic form $f$ is \textbf{primitive} if its coefficients generate the unit ideal of $R$.  The quadratic form $f$ corresponds to a $R$-lattice (necessarily free, since $R$ is a PID) $L$ in a quadratic space $(V,q)$ over $K$ on which $f$ is $R$-valued.
We say that $q$ is \textbf{maximal} if there is no $R$-lattice $M \supsetneq L$ on which $q$ is $R$-valued.
\\ \indent
We say that a quadratic form $f_{/R}$ is \textbf{ADC}
(cf. \cite{ADCI}, \cite{ADCII}) if for all $x \in R$, whenever there is $v \in K^n$ such that $f(v) = x$, there is $w \in R^n$ such that $f(w) = x$.  In other words, $f$ is ADC if and only if every element of $R$ that is $K$-represented by $f$ is moreover $R$-represented by $f$.  In symbols:
\[ D_f = D_{f_{/K}} \cap R. \]
For a prime element $p$ of $R$, let $K_p$ be the completion of $K$ with respect to the $p$-adic valuation $v_p$,
and let $R_p$ be the valuation ring.  We say that $f_{/R}$ is \textbf{locally ADC} if for all prime elements $p$, the form $f_{/R_p}$ is ADC.
\\ \\
A quadratic form $f_{/\Z}$ is \textbf{regular} if for all $m \in \Z$, if $f$ $\R$-represents $m$ and $\Z_p$-represents $m$ for
all primes $p$, then $f$ $\Z$-represents $m$.  A quadratic form  $f_{/\Z}$ is \textbf{almost regular} if the set
\[\{ m \in \Z \mid f \text{ represents } m \text{ over } \R \text{ and over } \Z_p \text{ for all primes $p$ but not over } \Z\}\] is finite.  (In this paper, the term ``almost'' in this paper will always mean ``all but finitely many.'')


\begin{thm}\mbox{ }
\label{UNTHM1}
\begin{itemize}
\item[a)] Let $f_{/R_p}$ be a quadratic form.  If $f$ is maximal, then it is ADC.
\item[b)] Let $f_{/R}$ be a quadratic form.  If $f$ is ADC, then it is locally ADC.
\item[c)] Let $f_{/\Z}$ be a quadratic form.  Then $f$ is ADC if and only if it is locally ADC and regular.
\end{itemize}
\end{thm}
\noindent
\begin{proof}
These are either all results from \cite{ADCI} or follow immediately from them.  \\
a) Combine \cite[Thm. 8]{ADCI} and \cite[Thm. 19]{ADCI}. \\
b) This is \cite[Cor. 17]{ADCI}.  \\
c) This is \cite[Thm. 25]{ADCI}.
\end{proof}
\noindent

\subsection{The density of an integral quadratic form}
Let $f \in \Z[t_1,\ldots,t_n]$ be an $n$-ary integral quadratic form.    As a special case of a definition given above, we put
\[ D_f := f(\Z^n) \setminus \{0\}. \]
Moreover, for $X \geq 1$, let
\[ D_f(X) \coloneqq D_f \cap [-X,X]. \]
We call $f$ positive, negative or indefinite according to whether $D_f$ is positive, negative or
indefinite.  If $f$ is negative, then $-f$ is positive and $D_f(X) = -D_{-f}(X)$, so negative forms
do not merit separate consideration.
\\ \\
One of the main problems in the area is, given $f$, to determine $D_f$ explicitly.  In general this is quite difficult.
Actually the situation is worse: it is not completely clear what ``determine $D_f$ explicitly'' means!  The set $D_f$ is recursive: there is an algorithm that, given $f$ and $m \in \Z$, determines whether $f$ represents $m$. 
Presumably we have in mind some \emph{finitistic} description of $D_f$.  This is possible e.g. if $f$ is known to be almost regular, but even so there are $14$ positive ternary forms that are known to be regular conditionally on GRH but not yet unconditionally \cite{LemkeOliver14}.

\subsection{Some motivating examples}

The point of departure of this work is the idea that it ought to be simpler to describe the \emph{size} of
$D_f$ rather than $D_f$ itself: namely in terms of the density
\[ \delta(f) \coloneqq \delta(D_f) \]
and -- when $\delta(f) = 0$  --
the asymptotic behavior of $D_f(X)$.  We consider several motivating examples.

\begin{example}\mbox{ }
\label{ISOEX}
\begin{itemize}
\item[a)] Let $f = t_1 t_2$, an indefinite form.  Then $D_f = \Z \setminus\{0\}$, so $\delta(f) = 1$. 
\item[b)] Let $f = t_1^2 - t_2^2$, an indefinite form.   Then 
\[ D_f = \{ m \in \Z^{\bullet} \mid v_2(m) \neq 1\}, \]
so $\delta(f) = \frac{3}{4}$.
\end{itemize}
\end{example}

\begin{example}[Fermat]
\label{FERMATLANDAUEX}
Let $f = t_1^2+t_2^2$, a positive form.  As Fermat knew, $f$ represents $m \in \Z^+$ if and only if 
$v_p(m)$ is even for all primes $p \equiv 3 \pmod{4}$. 
Thus for all primes $p \equiv 3 \pmod{4}$, $f$ does not represent any integer that is divisible by $p$ but
not by $p^2$.  Since \[\prod_{p \equiv 3 \pmod{4}} \left( 1-\frac{1}{p} + \frac{1}{p^2} \right) = 0,\] it
follows that $\delta(f) = 0$.  
\end{example}

\begin{example}[Gauss-Legendre]
Let $f = t_1^2 + t_2^2 + t_3^2$, a positive form.  Gauss and Legendre showed that $f$ represents $n \in \Z^+$ if and only if $n$ is \emph{not} of the form
$4^a (8b+7)$ for $a,b \in \N$.  It follows that
\[ \delta(f) = 1 - \frac{1+ 1/4 + 1/4^2 + \ldots}{8} = \frac{5}{6}. \]
\end{example}

\begin{example}\mbox{}
\label{LAGRANGEEX} 
\begin{itemize}
\item[a)] Let $f = t_1^2 + t_2^2 + t_3^2 + t_4^2$.  Lagrange showed that $D_f = \Z^+$, so $\delta(f) = 1$.
\item[b)] Let $f = 2023 t_1^2 + 2023 t_2^2 + 2023 t_3^2 + 2023 t_4^2$.  
Then $D_f = 2023 \Z^+$ so $\delta(f) = \frac{1}{2023}$.
\item[c)] For a prime number $p$, let $f_p = t_1^2 + p^2 t_2^2 + \ldots + p^2 t_4^2$.  
Then $f_p$ does not represent any integer with $p$-adic valuation $1$, so
\[ \delta(f_p) \leq 1 - \frac{1}{p} + \frac{1}{p^2} < 1. \]
\end{itemize}
\end{example}

\subsection{Results on density}

We observe a stark difference in behavior between Examples \ref{ISOEX} and \ref{FERMATLANDAUEX}.
This is due to the fact the form $t_1^2 + t_2^2$ is positive whereas the forms $t_1t_2$ and $t_1^2-t_2^2$
are not just indefinite but isotropic.    More generally, let $f_{/\Z} = At_1^2 + Bt_1 t_2 + Ct_2^2$ be a binary quadratic form of Discriminant\footnote{This Discriminant (note the capitalization) is defined only for binary quadratic forms.  On the other hand, for an 
$n$-ary quadratic form $q$ defined over a domain of characteristic different from $2$, we define the discriminant $\operatorname{disc}(f)$ to be the determinant of the Gram matrix of the associated bilinear form.} $\Delta \coloneqq B^2-4AC \neq 0$.  Then: $f$ is isotropic if and only if $\Delta = b^2$ for some $b \in \Z^+$.  (This follows, for instance, from \cite[Thm. I.3.2]{Lam}.)

\begin{thm}
\label{MAIN2THM}
\label{THM1.8}
Let $f_{/\Z}$ be a primitive binary quadratic form of Discriminant $\Delta$. 
\begin{itemize}
\item[a)] {\rm (Landau-Bernays  \cite{Landau08}, \cite{Bernays12})} Suppose $f$ is anisotropic.
There is a constant $\kappa > 0$ {\rm (}depending only on $\Delta${\rm )} such that as $X \ra \infty$ we have
\[ \#D_f(X) \sim \frac{ \kappa X}{\log^{1/2} X}. \] 
\item[b)] If $f$ is isotropic of Discriminant $\Delta = b^2$, then $\delta(f) > 0$.  More precisely, for a prime $p$,
let $a_p = v_p(\Delta)$.  Put
\[ \delta_2(f) \coloneqq \begin{cases} 1 & \text{if $a_2 = 0$},\\
\frac{3}{4} & \text{if $a_2 = 1$}, \\
\frac{2+2^{4-a_2} + 2^{5-a_2} + 2^{2-a_2}}{12} & \text{if $a_2 \geq 2$}, \end{cases} \]
and for $p > 2$, put
\[ \delta_p(f) \coloneqq  \frac{p + p^{1-a_p} + 2 p^{-a_p}}{2p+2}. \]
Then
\begin{equation}
\label{THM1.8EQ}
 \delta(f) = \prod_p \delta_p(f),
\end{equation}
the product extending over all prime numbers -- or equivalently, over all primes dividing $\Delta${\rm :} we have
$\delta_p(f) = 1$ if and only if $a_p = 0$.
\end{itemize}
\end{thm}

\begin{thm}
\label{THM1.9}
Let $n \geq 3$, and let $f_{/\Z}$ be an $n$-ary quadratic form.  Then $\delta(f) > 0$.
\end{thm}

\begin{thm}
\label{THM1.10}
Let $f_{/\Z}$ be an anisotropic ternary quadratic form.  Then $\delta(f) < 1$.
\end{thm}

\begin{remark}
It is immediate that for all $a \in \Z \setminus \{0\}$, we have $\delta( ax^2) = 0$.  So it follows from Theorems \ref{THM1.8}, \ref{THM1.9} and \ref{THM1.10} that the density $\delta(f)$ exists for any quadratic form 
$f_{/\Z}$.
\end{remark}

\begin{remark}
Theorem \ref{THM1.10} is essentially classical.  It follows from the Hasse-Minkowski theory that $f$ is anisotropic
at some prime number $p$ and that $f$ fails to $\Q_p$-represent $-\disc f$, and thus $f$ fails to represent all integers in a nonempty union of congruence classes modulo $p^4$ (if $p > 2$ we may replace $p^4$ by $p^2$). An extensive analysis of such local representation issues will be given later on. \\ \indent
 In particular, no positive definite ternary form is universal, a result that goes back at least to 1933 \cite{Albert33}.  See \cite{MO} and \cite{Doyle-Williams17} for further information on the history.  The latter work gives an explicit
congruence class of integers not represented by any given definite ternary $f$ using elementary methods.
\end{remark}

\begin{thm}
\label{THM1.12}
Let $n \geq 3$, and let $f_{/\Z}$ be a quadratic form.  If $n = 3$ we assume that $f$ is isotropic. 
\begin{itemize}
\item[a)] The following are equivalent:
\begin{itemize}
\item[(i)] $f$ is locally universal: $f_{/\Z_p}$ is universal for all primes $p$.
\item[(ii)] $f$ is locally ADC: $f_{/\Z_p}$ is ADC for all primes $p$. 
\item[(iii)] We have $\delta(f) = 1$.
\end{itemize}
\item[b)] If $f$ is maximal, then $\delta(f) = 1$.
\end{itemize}
\end{thm}
\noindent
These results invite us to consider the  ``inverse problem'' for densities of representation sets: which real numbers
arise as the density $\delta(f)$ of an 
 $n$-ary quadratic form?  Here are some results on this:

\begin{thm}[Rationality of densities]
\label{MAINQTHM}
\label{DENSITYISRATIONALTHM}
\label{THM1.13}
For $f_{/\Z}$ a quadratic form, we have $\delta(f) \in \Q$.
\end{thm}

\begin{thm}[Density of densities]
\label{MAIN3THM2}
\label{DDTHM}
\label{THM1.14}
Let $n \geq 3$, and let $r,s \in \N$ be such that $r+s = n$.  Let $\mathcal{S}_{r,s}$ be the set of $n$-ary
quadratic forms $f_{/\Z}$ with signature $(r,s)$.  Then the set
\[ \mathcal{D}_{r,s} \coloneqq \{\delta(f) \mid f \in \mathcal{S}_{r,s}\} \]
is dense in $[0,1]$.
\end{thm}

\begin{thm}
\label{THM1.15}
As $f_{/\Z}$ varies over all locally ADC ternary quadratic forms, the possible values of $\delta(f)$ comprise only $0\%$ of all rational numbers in $[0,1]$, in the sense of height.  Since maximal quadratic forms are locally ADC, the same holds as we vary over all maximal ternary $f_{/\Z}$.
\end{thm}
\noindent
We do not know whether for every $\delta \in (0,1) \cap \Q$ there is a ternary quadratic form $f_{/\Z}$ 
with $\delta(f) = \delta$.  

\subsection{The Near Regularity Theorem and the Density Hasse Principle}

The key to the proofs of the above results is a Hasse principle for $\delta(f)$. For a quadratic form $f_{/\Z}$, we say that that $m \in \Z$ is \textbf{locally represented} by $f$ if $f$ represents $m$ over $\Z_p$
for all primes $p$ and also over $\R$.  Let
$D_{f,\loc}$ be the set of integers $n$ that are locally represented by $f$.  Thus we have $D_f \subset D_{f,\loc}$, with equality if and only if $f$ is regular, whereas almost regularity
means that $D_f \setminus D_{f,\loc}$ is finite.

\begin{thm}[Tartakowsky-Kloosterman-Ross-Pall \cite{Ross-Pall46}]
\label{ROSSPALLTHM}
Let $f_{/\Z}$ be a positive quadratic form.  If $f_{/\Z_p}$ is isotropic for all primes $p$, then $f$ is almost regular.
\end{thm}
\noindent
When $n \geq 5$, every $n$-ary quadratic form over $\Z_p$ is isotropic, and thus we get:

\begin{cor}[Tartakowsky]
\label{TARTCOR}
For $n \geq 5$, every $n$-ary positive $f_{/\Z}$ is almost regular.
\end{cor}

\begin{remark}
Combining Corollary \ref{TARTCOR} and Theorem \ref{UNTHM1}a) gives: if $n \geq 5$, a maximal positive
 $n$-ary $f_{/\Z}$ represents all sufficiently large integers.
\end{remark}
\noindent
There are also positive quaternary forms $f_{/\Z}$ such that $f_{/\Z_p}$ is isotropic for all primes $p$, e.g. $t_1^2 + t_2^2 +
t_3^2 + n t_4^2$ where $n \in \Z^+$ is not of the form $4^a(8b+1)$.

\begin{thm}[Bochnak-Oh \cite{Bochnak-Oh08}]
\label{BOCHNAKOH}
Let $f_{/\Z}$ be an almost regular positive quaternary form.  Then the set of anisotropic primes of $f$ is either
empty or consists of a single prime $p \leq 37$.
\end{thm}


\begin{example}\mbox{ }
\begin{itemize}
\item[a)] The binary quadratic form $t_1^2 + 14 t_2^2$ is \emph{not} almost regular.  The set of prime numbers that it represents locally
has relative density $\frac{1}{4}$, whereas the set of prime numbers that it represents has relative density $\frac{1}{8}$.
\item[b)]  The ternary quadratic form $3t_1^2 + 4 t_2^2 + 9 t_3^2$ is \emph{not} almost regular (more details are given in Example \ref{ROUSEEX}).
\item[c)] The quadratic form $f = t_1^2 + t_2^2 + 7t_3^2 + 7t_4^2$ locally represents all positive integers. Since $f \equiv 0 \pmod{49}$ if and only if $t_{1} \equiv t_{2} \equiv t_{3} \equiv t_{4} \equiv 0 \pmod{7}$, then $f$ is anisotropic at $7$. The fact that $f$ does not represent $3$, $6$, $21$ or $42$ shows that $f$ fails to represent $3 \cdot 7^{k}$ or $6 \cdot 7^{k}$ for any integer $k \geq 0$. This shows that $f$ is \emph{not} almost regular (as was known to Watson \cite[p. 121]{Watson}). More effort shows that $f$ represents every positive integer except those of the form $3 \cdot 7^{k}$ or $6 \cdot 7^{k}$.
\end{itemize}
\end{example}

\noindent
So positive forms in fewer than five variables need not be almost regular.  However, there is a more permissive -- but analytically natural --
sense in which every quadratic form is ``nearly regular.''  Put
\[ D_{f,\loc}(X) \coloneqq D_{f,\loc} \cap [-X,X] \]
and
\[\mathcal{E}_f(X) \coloneqq D_{f,\loc}(X) \setminus D_f(X). \]
Let us also put 
\[ \delta_{\loc}(f) \coloneqq \delta(D_{f,\loc}). \]
For a prime number $p$, let $\mu_p$ be the unique Haar measure on $\Z_p$ with $\mu_p(\Z_p) = 1$.  For an $n$-ary integral quadratic form $f$, we define $\delta_p(f)$ to be $\mu_p( \{ f(x_1,\ldots,x_n) \mid (x_1,\ldots,x_n) \in \Z_p^n \})$.     (These are the quantities $\delta_p(f)$ that appear in Theorem \ref{MAIN2THM}.)
\\ \\
We have the following two key results.

\begin{thm}[Product Formula]
\label{PRODFORMTHM}
Let $f_{/\Z}$ be an $n$-ary quadratic form.  Then:
\begin{equation}
\label{PRODUCTFORMULA}
 \delta_{\loc}(f) = \prod_p \delta_p(f), 
\end{equation}
the product extending over all prime numbers.
\end{thm}
\noindent
We will see later that if $n \leq 2$ and $f$ is anisotropic, then $\prod_p \delta_p(f) = 0$.  In every other case we have 
$\delta_p(f) = 1$ for all sufficiently large $p$, so the product is actually finite.   Thus Theorem \ref{PRODFORMTHM}  implies 
that $\delta_{\loc}(f)$ exists.

\begin{thm}[Near Regularity]
\label{NEARREGTHM}
Let $f_{/\Z}$ be an $n$-ary quadratic form.\begin{itemize}
\item[a)] As $X \ra \infty$,
\begin{equation}
\label{EXCEPTIONALEQ}
\#\mathcal{E}_f(X) = o(D_f(X)).
\end{equation}
\item[b)] If $n \geq 3$ we have $\#\mathcal{E}_f(X) = O(\sqrt{X})$ {\rm (}whereas Theorem \ref{THM1.9} gives $D_f(X) \gg X${\rm )}.
\end{itemize}
\end{thm}
\noindent
These results have the following immediate consequence:

\begin{cor}[Density Hasse Principle]
\label{DENSITYHASSE}
For all $f_{/\Z}$, we have \[\delta(f) = \delta_{\loc}(f) = \prod_p \delta_p(f). \]
\end{cor}
\noindent
Corollary \ref{DENSITYHASSE} shifts the work of computing $\delta(f)$ to the more tractable setting of
quadratic forms over $\Z_p$.   In fact we carry out this local analysis first in \S 2.  In \S 3 we prove Theorem \ref{PRODFORMTHM}.
We prove Theorem \ref{NEARREGTHM} 
in \S 4, in fact with an explicit error bound that depends on the number of variables and whether $f$ is definite or indefinite.  
  In \S 5 we prove a globalization result, which is used in \S 6 along with 
Corollary \ref{DENSITYHASSE} to prove the remaining results.

\section{The local case}

\subsection{The representation table and the local representation measure}
In this paper, a \textbf{local field} is a complete discretely valued field $(K,v)$ of characteristic different from $2$, with valuation ring $R$, uniformizing element $\pi$,
and finite residue field $R/\pi R \cong \F_q$.  We say that $K$ (or $R$) is \textbf{nondyadic} if $q$ is odd and \textbf{dyadic} if
$q$ is even.  Let $\mathcal{S} \coloneqq K^{\times}/K^{\times 2}$ be the set of square classes of $K$. By a slight abuse of notation, we will also denote by $\mathcal{S}$ a chosen set of representatives for $K^{\times 2}$ in
$K^{\times}$.   \\ \indent
If $R$ is non-dyadic, we fix $r \in R^{\times}$ such that $r + \pi R \in \F_q^{\times} \setminus \F_q^{\times 2}$.  Then as coset 
representatives for $K^{\times 2}$ in $K$ we may take \cite[Thm. VI.2.2(1)]{Lam}
\[ \mathcal{S} = \{1,r,\pi,r\pi\}. \] If $R$ is dyadic then $\# \mathcal{S} = 2^{[K:\Q_2]+2}$ \cite[Cor. VI.2.23]{Lam}.  When $K = \Q_2$, as coset 
representatives for $\Q_2^{\times 2}$ in $\Q_2^{\times}$ we take
\[ \mathcal{S} = \{1,3,5,7,2,6,10,14\}. \]
For $n \geq 1$, let $f_{/R}$ be an $n$-ary quadratic form. In contrast to the global case, we can give a finitistic description of $D_f$ in all cases.  Namely,
for $s \in \mathcal{S}$, let $v_s$ be the minimal valuation $v(x)$ of an element $x \in D_f$ such that $x \in sK^{\times 2}$, or $\infty$
if no such element exists.  We refer to the assignment
\[ s \in \mathcal{S} \mapsto v_s \in \N \cup \{\infty\} \]
as the \textbf{representation table} of $f$.  Let $s \in \mathcal{S}$ with $v_s < \infty$.  Then there is $x \in D_f \cap sK^{\times 2}$
with valuation $v_s$; moreover, for any element $y \in sK^{\times 2}$ with $v(y) = v_s$ we have $y = u^2 x$ for some $u \in R^{\times}$, so $y \in D_f$.  Moreover, since $x \in D_f$, so is $\pi^{2k} x$ for all $k \geq 0$.  It follows that
\[ D_f \cap s K^{\times 2} = \{ y \in R^{\bullet} \mid y \in s K^{\times 2} \text { and } v(y) \geq v_s\}, \]
and thus knowing the representation table determines $D_f$.

\begin{remark}
Let $f_{/R}$ be an $n$-ary quadratic form.  Here is some information on the finiteness / infiniteness of the entries in the representation table of $f$. 
\begin{itemize}
\item[a)] Let $s \in \mathcal{S}$.  Then $v_s < \infty$ if and only if $f$ $K$-represents $s$.    
\item[b)] If $f$ is isotropic, then $f_{/K}$ is $K$-universal, so $v_s < \infty$ for all $s \in \mathcal{S}$.  
\item[c)] For a unary form we have $v_s < \infty$ for a unique $s \in \mathcal{S}$.  
\item[d)] If $f$ is an anisotropic binary form, then $f_{/K}$ $K$-represents precisely half of the elements of $\mathcal{S}$ \cite[p. 184, Exc. 8]{Lam}, so precisely half of the entries of the representation table of $f$ will be finite. If $K$ is nondyadic, we will shortly see by direct calculation that all $6$ possible pairs of square classes can arise this way.  
On the other hand, when $K = \Q_2$, the inequality ${8 \choose 4} > 8^2$ shows that not all $4$ element subsets of $\mathcal{S}$ 
arise as the set of square classes represented by an anisotropic binary quadratic form $f_{/\Q_2}$.   A straightforward asymptotic 
analysis shows as $a \coloneqq [K:\Q_2]$ tends to infinity, the number of $2^{a+1}$ element subsets of $\mathcal{S}$ arising 
as the set of square classes $K$-represented by an anisotropic binary quadratic form $f_{/K}$ divided by the number $\binom{2^{a+2}}{2^{a+1}}$ of $2^{a+1}$ element subsets of $\mathcal{S}$ approaches $0$.
\item[e)] If $f$ is an anisotropic ternary form, then $f_{/K}$ $K$-represents every $s \in \mathcal{S}$ except the class of 
$-\disc(f)$ and does not represent $-\disc(f)$ \cite[Cor. 2.15(2)]{Lam}.  So in this case exactly one entry of the representation table 
will be infinite, and by scaling any one anisotropic ternary form we see that this infinite entry can be any $s \in \mathcal{S}$.  
\item[f)] If $n \geq 4$, then $f_{/K}$ is universal \cite[Cor. 2.11 and Cor. 2.15(1)]{Lam}, so every entry in the representation table is finite.  Moreover, if $n \geq 5$ then $f$ is isotropic \cite[Thm. 2.12]{Lam}.
\end{itemize}
\end{remark}
\noindent
Let $\mu$ be the Haar measure on $R$ with unit mass.  Then we define the \textbf{local representation measure}
\[ \delta_v(f) \coloneqq \mu(D_f). \]
When $K = \Q_p$ we write $\delta_p(f)$ instead.  We also define
\[ \nu(R) \coloneqq [R^{\times}:R^{\times 2}], \]
so $\nu(R) = 2$ if $q$ is odd and $\nu(R) = 2^{[K:\Q_2]+1}$ is $q$ is even.
\\ \\
Let $a \in R^{\bullet}$ have valuation $v$.  If $u \coloneqq \frac{a}{\pi^v}$ then we get
\[ a R^{\times 2} = \pi^v (u R^{\times 2}). \]
For any measurable $Y \subset R$ and $x \in R^{\bullet}$, we have $\mu(x Y) = \frac{\mu (Y)}{q^{v(x)}}$.  It follows that
\[ \mu(R^{\times}) = \mu(R \setminus \pi R) = \frac{q-1}{q}, \]
\[ \mu(u R^{\times 2}) = \frac{q-1}{\nu(R) q} \]
and finally that
\[ \mu(a R^{\times 2}) = \frac{q-1}{\nu(R) q^{v+1}}. \]
Let $s \in \mathcal{S}$ be such that $v_s < \infty$, and choose $x \in D_f \cap s K^{\times 2}$ with $v(x) = v_x$.  Then
\begin{multline*} \mu(D_f \cap s K^{\times 2}) = \mu\left(\coprod_{i=0}^{\infty} \pi^{2i} x R^{\times 2}\right) \\ = \mu(x R^{\times 2})\left(1 + \frac{1}{q^2} + \frac{1}{q^4} + \ldots\right) = \frac{q^2}{q^2-1} \mu (x R^{\times 2}) = \frac{1}{\nu(R) q^{v_s-1}(q+1)}. \end{multline*}
When $v_s = \infty$, we take the expression $\frac{1}{\nu(R) q^{v_s-1}(q+1)}$ to be $0$.  This yields a formula for the local representation measure $\delta_v(f)$ in terms of the representation table:

\begin{thm}
\label{TABLEMEASURETHM}
Let $n \in \Z^+$, and let $f_{/R}$ be an $n$-ary quadratic form.  Then we have
\begin{equation}
\label{LOCALDENEQ}
\delta_v(f) = \sum_{s \in \mathcal{S}} \frac{1}{\nu(R) q^{v_s-1}(q+1)}.
\end{equation}
\end{thm}
\noindent
Here is an immediate consequence:

\begin{cor}
Let $n \in \Z^+$, and let $f_{/R}$ be an $n$-ary quadratic form.  Then:
\begin{itemize}
	\item[a)] We have $\delta_v(f) \in (0,1] \cap \Q$.  
\item[b)] We have $D(f) = R^{\bullet}$ if and only if $\delta_v(f) = 1$.
\end{itemize}
\end{cor}
\noindent
If $f_{/R}$ is a quadratic form, it is a routine task to compute its representation table and thus via Theorem \ref{TABLEMEASURETHM} 
its local representation measure $\delta_v(f)$.  However, for our later applications we want to solve a related \emph{inverse problem}: namely, for fixed $n \in \Z^+$ we would like to determine all possible representation tables for $n$-ary quadratic 
forms $f_{/R}$.  (What we actually need is the set of local representation measures, which we approach 
via the representation tables.)  For fixed $R$ and $n$, this is clearly a finite problem.  It is even a finite problem for fixed $R$ and varying $n$: one can show that the set of representation tables for $n$-ary forms $f_{/R}$ is independent of $n$ as long as $n$ is sufficiently large.  \\ \indent
Here we will solve this problem for every \emph{nondyadic} $R$ and show that the set of possible representation tables of $n$-ary 
forms is the same for all $n \geq 4$.  To do so we need only hand calculations.  
\\ \indent
For our applications to integral forms it would be desirable also to treat the case of $R = \Z_2$.  But this finite problem seems several orders of magnitude more difficult than the general nondyadic case.  Instead we will compute all possible local representation densities 
$\delta_2(f)$ for (i) isotropic binary forms $f_{/\Z_2}$ and (ii) ADC forms $f_{/\Z_2}$.

\subsection{Non-dyadics} Throughout this section we suppose that $R$ is non-dyadic, i.e., $\F_q = R/\pi R$ has characteristic different from $2$.  We order the square classes as $1,r,\pi,r\pi$.  If $f_{/R}$ is a quadratic form with representation table
$(\alpha,\beta,\gamma,\delta)$, then $\alpha,\beta$, if finite, are even, while $\gamma,\delta$ (if finite), are odd.
\\ \\
Since $R$ is non-dyadic, every quadratic form $f_{/R}$ is diagonalizable \cite[Thm. 8.1]{Gerstein}, so we may assume
\[ f = a_1 t_1^2 + \ldots + a_n t_n^2 \]
with $v(a_1) \leq \ldots \leq v(a_n)$.  Then $\frac{1}{a_1} f$ is a primitive quadratic form defined over $R$.
\begin{remark}
\label{PRIMITIVEREMARK}
The representation table for $f$ is easily determined from that of $\frac{1}{a_1} f$: indeed, ordering the square classes as $1,r,\pi,r \pi$, if the
representation table for $\frac{1}{a_1} f$ is $(\alpha,\beta,\gamma,\delta)$, then $a_1 = r^{\epsilon} u^2 \pi^k$ with $\epsilon \in \{0,1\}$,
$u \in R^{\times}$ and $k \geq 0$.  Then:
\begin{itemize}\item If $\epsilon = 0$ and $k$ is even, the  table for $f$ is $(\alpha+k,\beta+k,\gamma+k,\delta+k)$.
\item If $\epsilon = 0$ and $k$ is odd, the  table for $f$ is $(\gamma+k,\delta+k,\alpha+k,\beta+k)$.
\item If $\epsilon = 1$ and $k$ is even, the  table for $f$ is $(\beta+k,\alpha+k,\delta + k, \gamma + k)$.
\item If $\epsilon = 1$ and $k$ is odd, the  table for $f$ is $(\delta+k,\gamma+k,\beta+k,\alpha+k)$.
\end{itemize}
\end{remark}
\noindent
Thus we have reduced to the case of
\[ f = t_1^2 + a_2 t_2^2 + \ldots + a_n t_n^2. \]

\subsubsection{Binary forms} The discriminant of $t_1^2 + a_2 t_2^2$ is $a_2$, so as $a_2$ ranges over all $4$ elements of $\mathcal{S}$ we get four different $K$-isomorphism classes of binary forms.
\\ \\
\textbf{Case 2.1}: $f = t_1^2 - \pi^{2b} t_2^2$.  We claim that
the representation table is $(0,2b,2b+1,2b+1)$.  \\ \indent
The form $f$ is isotropic and thus $K$-universal.  Moreover, taking $t_1 \mapsto \pi^b t_1$, we find that $f$ represents $\pi^{2b}(t_1^2- t_2^2)$.  Since $2 \in R^{\times}$,
$t_1^2 -t_2^2$ is isomorphic to the hyperbolic plane $t_1t_2$ and thus is universal.  It follows that $f$ represents every element
with valuation at least $2b$.  \\
$\bullet$ Suppose $x_1^2 - \pi^{2b} x_2^2 = r \pi^{2b-2}$.  Then we must have $x_1 = \pi^{b-1} X_1$, so
\[ X_1^2 - \pi^2 x_2^2 = r, \]
and going mod $\pi$ gives $X_1^2 \equiv r \pmod{\pi}$, a contradiction. \\
$\bullet$ Suppose $x_1^2 - \pi^{2b} z_2^2 = u \pi^{2b-1}$, where $u \in \{1,r\}$.  Then $2v(x_1) = v(x_1^2) \geq 2b-1$ and thus the left hand side has valuation at least $2b$, contradiction.
\\ \\
\textbf{Case 2.2:} $f = t_1^2 - r \pi^{2b} t_2^2$.   We claim that the representation table is $(0,2b,\infty,\infty)$.  \\ \indent
The form $f$ is anisotropic and unimodular and thus $K$-represents
$1$ and $r$ but not $\pi$ and $r\pi$.   Moreover $f$ represents $\pi^{2b}(t_1^2 - r t_2^2)$.  The form $t_1^2 - r t_2^2$ $R$-represents both $1$ and $r$, by Hensel's Lemma.  So $f$ represents $r \pi^{2b}$.  \\
$\bullet$ Suppose $x_1^2 - r \pi^{2b} x_2^2 = r \pi^{2b-2}$.  Then $x_1 = \pi^{b-1} X_1$, so
\[ X_1^2 - r \pi^2 x_2^2 = r, \]
and going mod $\pi$ gives a contradiction.
\\ \\
\textbf{Case 2.3:} $f = t_1^2 + \pi^{2b+1} t_2^2$.  Then $f$ is anisotropic and $K$-represents $1$ and $\pi$ but not $r$ or $r \pi$.
We claim that the representation table is $(0,\infty,2b+1,\infty)$.  Clearly $f$ represents $\pi^{2b+1}$.  \\
$\bullet$ Suppose $x_1^2 + \pi^{2b+1} x_2^2 = \pi^{2b-1}$.  Then as above $\pi^{2b}$ divides the left hand side, contradiction.
\\ \\
\textbf{Case 2.4:} $f = t_1^2 + r \pi^{2b+1} t_2^2$.  Then $f$ is anisotropic and $K$-represents $1$ and $r \pi$ but not $r$ or $\pi$.
The representation table is $(0,\infty,\infty,2b+1)$; the computations are similar to those of Case 2.3.

\subsubsection{Ternary Forms} \textbf{} \\ \\ \noindent
\textbf{Case 3.1:} $f= t_1^2 - \pi^{2b} t_2^2 + c t_3^2$ with $v(c) \geq 2b$.  \\ We claim that the representation table is $(0,2b,2b+1,2b+1)$,
i.e., the same as for its binary subform $t_1^2 - \pi^{2b} t_2^2$, treated in Case 2.1.  All we need to show is that for $d = r \pi^{2b-2}$, $\pi^{2b-1}$, $r \pi^{2b-1}$, the equation
\[ x_1^2 - \pi^{2b} x_2^2 + c x_3^2 = d \]
is not solvable for $(x_1,x_2,x_3) \in R^3$.  The argument of Case 2.1 carries over almost verbatim.
\\ \\
\textbf{Case 3.2.1:} $f= t_1^2 - r \pi^{2b} t_2^2 - \pi^{2c} t_3^2$ with $b \leq c$. \\  We claim that the representation table is $(0,2b,2c+1,2c+1)$.  Since $f$ represents $\pi^{2c}(t_1^2-t_3^2)$, it represents all elements with valuation at least $2c$
and thus $\pi^{2c+1}$ and $r \pi^{2c+1}$.  If for $d \in R$ the equation $x_1^2 - r \pi^{2b} x_2^2 - \pi^{2c} x_3^2 = d$
is solvable, so is $x_1^2 - r \pi^{2b} x_2^2 - \pi^{2b} x_3^2 = d$, so we may assume $c = b$.  \\
$\bullet$ If $x_1^2 - r \pi^{2b} x_2^2 - \pi^{2c} x_3^2 = r \pi^{2b-2}$ is solvable, then so is
$x_1^2 - r \pi^{2b} x_2^2 - \pi^{2b} x_3^2 = r \pi^{2b-2}$.  Then $\pi^{b-1} \mid x$, so
\[X_1^2 - r \pi^2 x_2^2 - \pi^2 x_3^2 = r, \] and going modulo $\pi$ gives a contradiction.  \\
$\bullet$ If $x_1^2 - r \pi^{2b} x_2^2 - \pi^{2c} x_3^2 = u \pi^{2c-1}$ is solvable, then so is
\[ x_1^2 - r x_2^2 - \pi^{2c} x_3^2 = u \pi^{2c-1}. \]
Since $x_1^2 - r x_2^2$ is anisotropic modulo $\pi$, we get $x_1 = \pi X_1$, $y_1 = \pi Y_1$ and
\[ \pi X_1^2 - r \pi X_2^2 - \pi^{2c-1} x_3^2 = u \pi^{2c-2}. \]
If $c = 1$, the left hand side is divisible by $\pi$ and the right hand side is not, contradiction.  Otherwise:
\[ X_1^2 - r X_2^2 - \pi^{2c-2} x_3^2 = u \pi^{2c-3}, \]
and by induction on $c$ we get a contradiction.
\\ \\
\textbf{Case 3.2.2:} $f=  t_1^2 - r \pi^{2b} t_2^2 - r \pi^{2c} t_3^2$ with $b \leq c$. \\  The representation table is
$(0,2b,2c+1,2c+1)$; the computations are similar to those of Case 3.2.1.
\\ \\
\textbf{Case 3.2.3:} $f = t_1^2 - r \pi^{2b}t_2^2 + \pi^{2c+1} t_3^2$ with $b \leq c$. \\ We claim that the representation table is
$(0,2b,2c+1,\infty)$.  Since $f$ is anisotropic, it does not represent $-\disc f \pmod{K^{\times 2}} = r \pi \pmod{K^{\times 2}}$.
As above it represents $r \pi^{2b}$ and clearly it represents $\pi^{2c+1}$.  \\
$\bullet$ If $x^2 - r \pi^{2b} y^2 + \pi^{2c+1} z^2 = r \pi^{2b-2}$ is solvable, then so is
\[ x^2 - r \pi^{2b} y^2 + \pi z^2 = r \pi^{2b-2}. \]
If $b = 1$ then going modulo $\pi$ gives a contradiction.  Otherwise we may take $x = \pi X$, getting
\[ \pi X^2 - r \pi^{2b-1} y^2 + z^2 = r \pi^{2b-3} \]
and then $z = \pi Z$, getting
\[ X^2 - r \pi^{2b-2} y^2 + \pi Z^2 = r \pi^{2b-4}, \]
and by induction on $b$ we get a contradiction. \\
$\bullet$ If $x^2 - r \pi^{2b}y^2 + \pi^{2c+1} z^2 = \pi^{2c-1}$ is solvable, then so is
\[ x^2 - r y^2 + \pi^{2c+1} z^2 = \pi^{2c-1}. \]
Going modulo $\pi$ shows we can take $x = \pi X$, $y = \pi Y$, getting
\[ \pi X^2 - r \pi Y^2 + \pi^{2c} z^2 = \pi^{2c-2}. \]
If $c= 1$ then going modulo $\pi$ gives a contradiction; else we get
\[ X^2 - r Y^2 + \pi^{2c-1} z^2 = \pi^{2c-3}, \]
and by induction on $c$ we get a contradiction.
\\ \\
\textbf{Case 3.2.4:} $f = t_1^2 - r \pi^{2b} t_2^2 + r \pi^{2c+1} t_3^2$ with $b \leq c$. \\
The representation table is $(0,2b,\infty,2c+1)$; the computations are similar to those of Case 3.2.3.
\\ \\
\textbf{Case 3.3.1:} $f = t_1^2 + \pi^{2b+1} t_2^2 - \pi^{2c} t_3^2$ with $b< c$.  \\
We claim that the representation table is $(0,2c,2b+1,2c+1)$.  By Case 2.1, $f$ represents $r \pi^{2c}$; clearly it represents
$\pi^{2b+1}$.  Also $f$ represents $\pi^{2c}(t_1^2-t_3^2)$ so represents $r \pi^{2c+1}$.  \\
$\bullet$ If $f = x_1^2 + \pi^{2b+1} y^2 - \pi^{2c} z^2 = r \pi^{2c-2}$ is solvable, then so is
\[ x^2 + \pi y^2 - \pi^{2c} z^2 = r \pi^{2c-2}. \]
If $c = 1$, then going modulo $\pi$ yields a contradiction.  Else we may take $x = \pi X$, getting
\[ \pi X^2 + y^2 - \pi^{2c-1} z^2 = r \pi^{2c-3} \]
and then $y = \pi Y$, getting
\[ X^2 + \pi Y^2 - \pi^{2c-2} z^2 = r \pi^{2c-4}, \]
and by induction on $c$ we get a contradiction.  \\
$\bullet$ If $f = x^2 + \pi^{2b+1}y^2 - \pi^{2c} z^2 = \pi^{2b-1}$ is solvable, then we may take $x = \pi^b X$, getting
\[  \pi X^2 + \pi^2 y^2 - \pi^{2c-2b+1} z^2 = 1, \]
and going modulo $\pi$ gives a contradiction. \\
$\bullet$ If $f = x_1^2 + \pi^{2b+1} y^2 - \pi^{2c} z^2 = r \pi^{2c-1}$ is solvable, then so is
\[ f = x^2 + \pi y^2 - \pi^{2c} z^2 = r \pi^{2c-1}. \]
We may take $x = \pi X$, getting
\[ \pi X^2 + y^2 - \pi^{2c-1} z^2 = r \pi^{2c-2}. \]
If $c = 1$, then going modulo $\pi$ gives a contradiction.  Else we can take $y = \pi Y$, getting
\[ X^2 + \pi y^2 - \pi^{2c-2} z^2 = r \pi^{2c-3}, \]
and by induction on $c$ we get a contradiction.
\\ \\
\textbf{Case 3.3.2:} $f = t_1^2 + \pi^{2b+1} t_2^2 - \pi^{2c+1} t_3^2$ with $b \leq c$.  \\
The representation table is $(0,2c+2,2b+1,2c+1)$;  the computations are similar to those of Case 3.3.1.
\\ \\
\textbf{Case 3.3.3:} $f = t_1^2 + \pi^{2b+1} t_2^2 - r \pi^{2c} t_3^2$ with $b < c$. \\
We claim that the representation table is $(0,2c,2b+1,\infty)$.  \\ \indent
Since $f$ is anisotropic, it does not represent
\[-\disc f \pmod{K^{\times 2}} = r\pi \pmod{K^{\times 2}}. \]  Clearly $f$ represents $\pi^{2b+1}$; it also represents
$\pi^{2c}(t_1^2 - r t_3^2)$ hence $\pi^{2c} r$.  \\
$\bullet$ If $f = x^2 + \pi^{2b+1} y^2 - r \pi^{2c} z^2 = r \pi^{2c-2}$ is solvable, then so is
\[ x^2 + \pi y^2 - r \pi^{2c} z^2 = r \pi^{2c-2}. \]
If $c = 1$, then going modulo $\pi$ gives a contradiction.  Else we may take $x = \pi X$, getting
\[ \pi X^2 + y^2 - r \pi^{2c-1} z^2 = r \pi^{2c-3}, \]
and then $y = \pi Y$, getting
\[ X^2 + \pi Y^2 - r \pi^{2c-2} z^2 - r \pi^{2c-4}, \]
and by induction on $c$ we get a contradiction. \\
$\bullet$ If $f = x^2 + \pi^{2b+1} y^2 - r \pi^{2c} z^2 = \pi^{2b-1}$ is solvable, then so is
\[ x^2 + \pi^{2b+1} y^2 - r \pi^{2b} z^2 = \pi^{2b-1}. \]
We may take $x = \pi X$, getting
\[ \pi X^2 + \pi^{2b} y^2 - r \pi^{2b-1} z^2 = \pi^{2b-2}. \]
If $b = 1$, then going modulo $\pi$ gives a contradiction.  Else we get
\[ X^2 + \pi^{2b-1} y^2 - r \pi^{2b-2} z^2 = \pi^{2b-3}, \]
and by induction on $b$ we get a contradiction.
\\ \\
\textbf{Case 3.3.4:} $f = t_1^2 + \pi^{2b+1} t_2^2 - r \pi^{2c+1} t_3^2$ with $b \leq c$. \\
The representation table is $(0,\infty,2b+1,2c+1)$; the computations are similar to those of Case 3.3.3.
\\ \\
\textbf{Case 3.4.1:} $f = t_1^2 + r \pi^{2b+1} t_2^2 - \pi^{2c} t_3^2$ with $b < c$.  \\
The representation table is $(0,2c,2c+1,2b+1)$; the computations are similar to those of Case 3.3.1.
\\ \\
\textbf{Case 3.4.2:} $f= t_1^2 + r \pi^{2b+1} t_2^2 - r \pi^{2c+1} z^2$ with $b \leq c$. \\
The representation table is $(0,2c+2,2c+1,2b+1)$; the computations are similar to those of Case 3.3.1.
\\ \\
\textbf{Case 3.4.3:} $f = t_1^2 + r \pi^{2b+1} t_2^2 - r \pi^{2c} t_3^2$ with $b < c$.  \\
The representation table is $(0,2c,\infty,2b+1)$;  the computations are similar to those of Case 3.3.3.
\\ \\
\textbf{Case 3.4.4:} $f= t_1^2 + r \pi^{2b+1} t_2^2 - \pi^{2c+1} t_3^2$ with $b \leq c$.  \\
The representation table is $(0,\infty,2c+1,2b+1)$;  the computations are similar to those of Case 3.3.3.   
\\ \\
From this we  find that for a ternary form $x^2 + by^2 + cz^2$ the possible representation tables are precisely the
following: for any $b,c \in \N$,
\begin{multline*} (0,\infty,2b+1,2c+1),  (0,2b,\infty,2c+1),  (0,2b,2c+1,\infty),  \\ 
(0,2b,2b+1,2b+1),  (0,2b,2c+1,2c+1),  (0,2b,2c+1,2b+1), \end{multline*}
and also:
\[ \forall b,c \in \N \text{ with } b \leq c, \ (0,2b,2c+1,2c+1), \ (0,2c+2,2c+1,2b+1) \]
\[ \forall b,c \in \N \text{ with } c < b, \ (0,2b,2c+1,2b+1), \ (0,2b,2b+1,2c+1) \]
\[ \forall b,c \in \N \text{ with } c \leq b, \ (0,2b+2,2c+1,2b+1), \ (0,2b+2,2b+1,2c+1). \]

\subsubsection{Quaternary forms} Let $f_{/R}$ be a quaternary form.  Then $f$ is $K$-universal, so $f$ $R$-represents elements of every $K$-adic square class, and thus the representation table is of
the form $(2a,2b,2c+1,2d+1)$ for $a,b,c,d \in \N$.  We claim that all of these representation tables actually occur.  As in Remark
\ref{PRIMITIVEREMARK}, via scaling it is enough to show that all representation tables $(0,2b,2c+1,2d+1)$ occur.  This is accomplished by the following calculations.
\\ \\
\textbf{Case 4.1:} $f: t_1^2 - r \pi^{2b} t_2^2 + r \pi^{2d+1} t_3^2 - \pi^{2c+1} t_4^2$ with $d \leq c$. \\
We claim that the representation table is $(0,2b,2c+1,2d+1)$.  \\ \indent By Case 2.2, $f$ represents $r \pi^{2b}$.  Clearly $f$ represents
$r \pi^{2d+1}$.  Moreover $f$ represents $\pi^{2c+1}(rt_3^2 - t_4^2)$ hence it represents $\pi^{2c+1}$.  \\
$\bullet$ If $f = x^2 - r \pi^{2b} y^2 + r \pi^{2d+1} z^2 - \pi^{2c+1} w^2 = r \pi^{2b-2}$ is solvable,
then so is
\[ x^2 - r \pi^{2b} y^2 + r \pi z^2 - \pi w^2 = r \pi^{2b-2}. \]
If $b = 1$, then going modulo $\pi$ gives a contradiction.  Else we may take $x = \pi X$, getting
\[ \pi X^2 - r \pi^{2b-1} y^2 + r z^2 - w^2 = r \pi^{2b-3}. \]
Going modulo $\pi$ shows that we may take $z = \pi Z$ and $w = \pi W$, getting
\[ X^2 - r \pi^{2b-2} y^2 + r \pi Z^2 - \pi W^2 = r \pi^{2b-4}, \]
and by induction on $b$ we get a contradiction.  \\
$\bullet$ If $f = x^2 - r \pi^{2b} y^2 + r \pi^{2d+1} z^2 - \pi^{2c+1} w^2 = \pi^{2c-1}$ is solvable,
then so is
\[ x^2 - r y^2 + r \pi z^2 - \pi^{2c+1} w^2 = \pi^{2c-1}. \]
Going modulo $\pi$ shows that we may take $x = \pi X$ and $y = \pi Y$, getting
\[ \pi X^2 - r \pi Y^2 + r z^2 - \pi^{2c} w^2 = \pi^{2c-2}. \]
If $c = 1$, then going modulo $\pi$ gives a contradiction.  Else we may take $z = \pi Z$, getting
\[ X^2 - r Y^2 + r \pi Z^2 - \pi^{2c-1} w^2 = \pi^{2c-3}, \]
and by induction on $c$ we get a contradiction. \\
$\bullet$ If $f = x^2 - r \pi^{2b} y^2 + r \pi^{2d+1} z^2 - \pi^{2c+1} w^2 = r \pi^{2d-1}$ is solvable,
then so is
\[x^2 - r y^2 + r \pi^{2d+1} z^2 - \pi^{2d+1} w^2 = r \pi^{2d-1}. \]
Going modulo $\pi$ shows that we may take $x = \pi X$ and $y = \pi Y$, getting
\[ \pi X^2 - r \pi Y^2 + r \pi^{2d} z^2 - \pi^{2d} w^2 = r \pi^{2d-2}. \]
If $d = 1$, then going modulo $\pi$ gives a contradiction.  Else we get
\[ X^2 - r Y^2 + r \pi^{2d-1} z^2 - \pi^{2d-1} w^2 = r \pi^{2d-3}. \]
Going modulo $\pi$ shows that we may take $X = \pi \mathcal{X}$ and $Y = \pi \mathcal{Y}$, getting
\[ \pi \mathcal{X}^2 - r \pi \mathcal{Y}^2 + r \pi^{2d-2} z^2 - \pi^{2d-2} w^2 = r \pi^{2d-4}, \]
and by induction on $d$ we get a contradiction.
\\ \\
\textbf{Case 4.2:} $f: t_1^2 - r \pi^{2b} t_2^2 + \pi^{2c+1} t_3^2 - r \pi^{2d+1} t_4^2$ with $c \leq d$.  \\
The representation table is $(0,2b,2c+1,2d+1)$; the computations are similar to those of Case 4.1.

\subsubsection{$n \geq 5$} Suppose that $f_{/R}$ is an $n$-ary quadratic form with $n \geq 5$.  
\\ \\
\textbf{Case 5.1:} Let $b,c,d \in \N$ with $d \leq c$, and put $A \coloneqq \max(2b-1,2c,2d)$.  Let $f: t_1^2 - r \pi^{2b} t_2^2 + r \pi^{2d+1} t_3^2 - \pi^{2c+1} t_4^2 + \pi^A (t_5^2 + \ldots + t_n^2)$.  Revisiting the argument
of Case 4.1, we see that the representation table remains $(0,2b,2c+1,2d+1)$.
\\ \\
\textbf{Case 5.2:} Let $b,c,d \in \N$ with $c \leq d$, and put $A \coloneqq \max(2b-1,2c,2d)$.  Let $f: t_1^2 - r \pi^{2b} t_2^2 + \pi^{2c+1} t_3^2 - r \pi^{2d+1} t_4^2 + \pi^A (t_5^2 + \ldots + t_n^2)$.  Revisiting the argument
of Case 4.1, we see that the representation table remains $(0,2b,2c+1,2d+1)$.

\subsubsection{A consequence}

\begin{thm}
Let $R$ be a complete DVR with residue field of finite odd cardinality $q$.  
Let $f_{/R}$ be an $n$-ary anisotropic ADC form.  
\begin{enumerate}
	\item[a)] If $n = 2$, then $\delta_v(f) \in \{\frac{q}{q+1}, \frac{1}{2}, \frac{1}{q+1}\}$, and all of these occur.
	\item[b)] If $n = 3$, then $\delta_v(f) \in \{\frac{q+2}{2q+2}, \frac{2q+1}{2q+2}\}$, and both of these occur. 
	\item[c)] If $n \geq 4$, then $\delta_v(f) = 1$.
\end{enumerate}
\end{thm}

\begin{proof}
We have $v_1,v_r \in \{0,\infty\}$ and $v_{\pi}, v_{r \pi} \in \{1,\infty\}$.  So if $f$ represents $a$ out
of the square classes $\{1,r\}$ and $b$ out of the square classes $\{\pi,r\pi\}$ then by (\ref{LOCALDENEQ}) we have
\[ \delta_v(f) = \frac{aq + b}{2q+2}. \]
a) If $n = 2$, then we know that $a + b = 2$, so $(a,b) \in \{ (2,0), (1,1), (0,2)\}$, which leads, respectively,
to $\delta_v(f) = \frac{q}{q+1}$, $\delta_v(f) = \frac{1}{2}$ and $\delta_v(f) = \frac{1}{q+1}$.  Moreover,
there are $6 = {4 \choose 2}$ inequivalent anisotropic binary forms over $K$ and 
every pair of square classes is represented by exactly one of these forms.  So all of these densities occur. 

\vskip 0.1in

\noindent b) If $n = 3$, then $f$ represents all square classes except $-\disc f$.  Thus we have $(a,b) \in \{ (2,1), (1,2)\}$, which leads, respectively, to $\delta_v(f) \in \{ \frac{2q+1}{2q+2}, \frac{q+2}{2q+2} \}$.  Moreover, $-\disc f$
can be any square class, just by scaling any one anisotropic ternary form.  So all of these densities occur. 

\vskip 0.1in

\noindent c) If $n \geq 4$ then $f$ is $K$-universal and ADC hence $R$-universal.  So $\delta_v(f) = 1$.
\end{proof}

\subsection{ADC forms over $\Z_2$}
Let $f_{/\Z_2}$ be an $n$-ary ADC form.  The ADC condition implies that for all $s \in \{1,3,5,7\}$ we have $v_s \in \{0,\infty\}$, while
for all $s \in \{2,6,10,14\}$ we have $v_s \in \{1,\infty\}$.  Thus if $f$ represents $a$ of the square classes $\{1,3,5,7\}$ and $b$ of the square classes $\{2,6,10,14\}$, by (\ref{LOCALDENEQ}) we have
\[ \delta_v(f) = \frac{2a+b}{12}. \]
If $n \geq 4$ then $f$ is $\Q_2$-universal and ADC, so it is $\Z_2$-universal and $\delta_2(f) = 1$.  Moreover
if $f$ is isotropic and ADC, then it is $\Z_2$-universal and $\delta_2(f) = 1$.  So the nontrivial cases (among ADC forms) are
when $n \in \{2,3\}$ and $f$ is anisotropic.
\\ \\
Recall also that in the complete local case maximal lattices are ADC, so every $\Q_2$-isomorphism class of quadratic forms yields
at least one $\Z_2$-isomorphism class of ADC forms.

\subsubsection{Binary forms} Let $f_{/\Z_2}$ be an anisotropic ADC binary form.  By \cite{Lam}, $f_{/\Q_2}$ represents precisely $4$
of the $8$ elements of $\mathcal{S}$.  Thus $a+b = 4$.    One sees -- e.g. by a brute force search of the $36$ binary forms $ax^2 + by^2$ obtained by letting
$a$ and $b$ run through unordered pairs of square classes in $\Q_2$ -- that the pairs $(1,3)$ and $(3,1)$ do not occur.  As for the others:
\\ \\
$2x^2 + 2xy + 2y^2$ is ADC and $\Q_2$-represents $2,6,10,14$, so $(a,b) = (0,4)$, and $\delta_2(f) = \frac{1}{3}$. \\
$x^2+y^2$ $\Q_2$-represents $1,2,5,10$, hence $(a,b) = (2,2)$, and $\delta_2(f) = \frac{1}{2}$. \\
$x^2 + xy + y^2$ is ADC and $\Q_2$-represents $1,3,5,7$, hence $(a,b) = (4,0)$, and $\delta_2(f) = \frac{2}{3}$.

\subsubsection{Ternary forms} Let $f_{/\Z_2}$ be an anisotropic ADC ternary form.  Then $f$
fails to represent the $\Q_2$-square class of $-\disc f$ and represents all other square classes.  By scaling we see that $-\disc f$ can be any square class.  So the possibilities are
\[ \delta_2(f) = \frac{5}{6}, \ \delta_2(f) = \frac{11}{12}. \]

\subsection{Isotropic binary forms over $\Z_2$} Let
\[ f(t_1,t_2) = A t_1^2 + B t_1t_2 + C t_2^2 \]
be a primitive isotropic binary form over $\Z_2$ with discriminant $\Delta = B^2-4AC$.  Then $\Delta$ is a square
in $\Z_2$, so $a \coloneqq \frac{1}{2}v_2(\Delta) \in \N$.  We will compute $\delta_2(f)$ in terms of $a$.  Since every binary
form over $\Z_2$ is either diagonalizable or $\Z_2$-equivalent to $2^a(t_1^2 + t_1 t_2 + t_2^2)$ or
$2^a t_1 t_2$ for some $a \in \N$ \cite[Lemma 8.4.1]{Cassels}, and since $f$ is primitive with square discriminant, $f$ is isomorphic to either
$t_1 t_2$ or $u (t_1^2 - 2^{2a-2} t_2^2)$ for $u \in \Z_2^{\times}$.  The former case occurs if and only if $a = 0$, and
in this case clearly $\delta_2(f) = 1$.    In the latter case we have $\delta_2(f) = \delta_2(t_1^2 - 2^{2a-2} t_2^2)$.
If $a = 1$ then $\delta_2(f) = \delta_2(t_1^2-t_2^2)$.  It is easy to see that $f$ represents precisely the elements
of $\Z_2$ with valuation different from $1$, so $\delta_2(f) = \frac{3}{4}$. Now we suppose that $a \geq 2$.
Ordering the $\Q_2$ square classes as $1,3,5,7,2,6,10,14$, we claim that the representation table of $f$ is
\[ (0,2a-2,2a-4,2a-2,2a+1,2a+1,2a+1,2a+1). \]
First suppose that $a = 2$.  Then one easily sees that $t_1^2-4t_2^2$ has representation table
$(0,2,0,2,5,5,5,5)$.  Now suppose $a \geq 3$.  Clearly $t_1^2-2^{2a-2} t_2^2$ represents $1$; reducing modulo $8$ shows that it does not represent any of $2,3,5,6,7,10,14$.  Now, for $w \in \Z_p$, suppose that there are
$x,y \in \Z_2$ such that $x^2 - 2^{2a-2} y^2 = 4w$.  Then we may take $x = 2X$ and get
\[ X^2- 2^{2a-4} y^2 = w, \]
and conversely if $t_1^2 - 2^{2a-4} t_2^2$ represents $w$ then $t_1^2 - 2^{2a-2} t_2^2$ represents $4w$.
So if the representation table for $t_1^2-2^{2a-4} t_2^2$ is \[(0,2a-4,2a-6,2a-4,2a-1,2a-1,
2a-1,2a-1)\] then the representation table for $t_1^2 - 2^{2a-2} t_2^2$ is \[ (0,2a-2,2a-4,2a-2,2a+1,2a+1,2a+1,2a+1), \]
 and the claim follows by induction on $a$.
\\ \\
Combining the above analysis with (\ref{LOCALDENEQ}), we get the following result.

\begin{prop}
\label{PROP2.5}
Let $f_{/\Z_2}$ be a primitive isotropic binary quadratic form of Discriminant $\Delta$, and let $a = \frac{1}{2}v_2(\Delta)$.  Then: \\
a) If $a = 0$, then $\delta_2(f) = 1$.  \\
b) If $a = 1$, then $\delta_2(f) = \frac{3}{4}$. \\
c) If $a \geq 2$, then $\delta_2(f) = \frac{2+ 2^{4-2a} + 2^{5-2a} + 2^{2-2a}}{12}$.
\end{prop}

\subsection{A consequence}
The following is an immediate consequence of our results.

\begin{thm}
\label{LOCALADCTHM}
Let $f_{/\Z_p}$ be an $n$-ary ADC form (e.g. a maximal lattice).  Suppose moreover that either $n = 3$ or $p > 2$.  Then either $\delta_p(f) = 1$ or $\delta_p(f)$ is a rational number
with negative $2$-adic valuation.
\end{thm}

\subsection{Some examples}

\begin{example}
\label{INTERESTINGEXAMPLE}
Let $p > 2$ and let $f = x^2 + py^2 - pz^2$.  The binary subform $py^2-pz^2 = p(y^2-z^2)$ is isotropic, hence so is $f$.
Since $y^2-z^2 \cong yz$ is universal, the form $f$ represents every square class with valuation at least $1$.  Clearly it also
represents $\Z_p^{\times 2}$, but it does not represent the unit non-residue $r$: for all $x,y,z \in \Z_p$, $x^2 + py^2 - pz^2 \pmod{p}
\in \F_p^{\times 2} \cup \{0\}$.  We conclude
\[ \delta_p(f) = 1 - \frac{p-1}{2p} = \frac{p+1}{2p}. \]
Thus $v_2(\delta_p(f))$ is non-negative for all $p > 2$, is strictly positive if and only if $p \equiv 3 \pmod{4}$ and is arbitrarily large on a set
of primes of positive relative density.
\end{example}


\begin{example}
Let $n \geq 3$, $p > 2$ and put
\[ f = t_1^2 + p^2 t_2^2 - p^2 t_3^2 - \ldots  - p^2 t_n^2. \]
Then the subform $p^2 t_2^2 - p^2 t_3^2$ is isotropic and represents all $x$ with $v(x) \geq 2$.  Moreover $f$ does represent
$1$ and does not represent any of $r$, $p$, $rp$.  So
\[ \delta_p(f) = 1-\frac{p-1}{2p} - \frac{p-1}{2p^2} - \frac{p-1}{2p^2} = \frac{p^2-p+2}{2p^2}. \]
\end{example}

\begin{example}
Let $n \geq 4$ and let $p > 2$.  Choose $a_1,a_2 \in \Z_p^{\times}$ such that $-a_1a_2 \equiv r \pmod{\Z_p^{\times 2}}$.
Let $f_0 = a_1 t_1^2 + a_2 t_2^2 + p t_3^2$ and
\[ f(t_1,\ldots,t_n) \coloneqq f_0(t_1,t_2,t_3) + p^2 t_4^2 + \ldots + p^2 t_n^2. \]
The ternary subform $f_0$ is maximal and anisotropic, so represents precisely
those elements of $\Z_p$ that do not lie in the $\Q_p$-adic square class of $rp$.  If $x_1,\ldots,x_n \in \Z_p$ are such that
\[ f(x_1,\ldots,x_n) = a_1 x_1^2 + a_2 x_2^2 + p x_3^2 + p^2 x_4^2 + \ldots + p^2 x_n^2 = rp, \]
then $a_1 x_1^2 + a_2 x_2^2 \equiv 0 \pmod{p}$ and thus there are $X_1,X_2 \in \Z_p$ with $x_1  =p X_1$, $x_2 = p X_2$.
Making this substitution and simplifying, we get
\[ p a_1 X_1^2 + pa_2 X_2^2 + x_3^2 + p x_4^2 + \ldots + p x_n^2 = r.\]
Reducing mod $p$ gives $x_3^2 \equiv r \pmod{p}$: contradiction.  So $f$ does not represent $rp$.  \\ \indent
Since $f_0$ represents $p^2(rp-1)$ and $p^2 t_4^2 + \ldots + p^2 t_n^2$ represents $p^2$, $f$ represents
$p^2(rp-1) + p^2 = rp^3$, hence also $rp^5$, $rp^7$ and so forth.  Thus $f$ represents everything but the $\Z_p$-square
class of $rp$, so
\[ \delta_p(f) = 1-\frac{p-1}{2p^2}. \]
\end{example}

\section{Proof of the Product Formula}
\noindent
Suppose first that $n=1$. Then for all but finitely many primes $p$, we have \[\delta_p(f)= \frac{p-1}{2p} \bigg{(}1+1/p^2+1/p^4+\dots \bigg{)} = \frac{p}{2(p+1)} < 1/2, \] so $\prod_{p} \delta_p(f) = 0$. Clearly, $\delta(f)=0$ in this case as well, so the formula \eqref{PRODUCTFORMULA} holds. If $n=2$ and $f_{/\Z}$ is anisotropic, again both sides of \eqref{PRODUCTFORMULA} vanish, as shown in \S \ref{sec:anisotropicbinary}. 

Thus, we may assume that either $n\ge 3$, or that $n=2$ and $f_{/\Z}$ is isotropic. In either case, we may fix a finite set of primes $S$ such that $f_{/\Z_p}$ is universal for all $p\notin S$. For if $n \geq 3$, we have seen already that $\delta_p = 1$ for all odd primes not dividing the discriminant. When $n=2$ and $f$ is isotropic, then for all odd primes not dividing the discriminant, $f_{/\Z_p}$ is isomorphic to the hyperbolic plane and hence universal. Since $\delta_f = \delta_{\loc}(f)$ (the Density Hasse Principle), to complete the proof of \eqref{PRODUCTFORMULA} it will suffice to show that 
\begin{equation}\label{eq:densityhasse2} \delta_{\loc}(f) = \prod_{p\in S} \delta_p(f). \end{equation}

Let $p\in S$. Asking that the nonzero integer $n$ be represented by $f_{/\Z_{p}}$ is equivalent to requiring that $n$ lie in one of the $\Z_p$-square classes represented by $f$. For $K \in \Z^{+}$, let $U_{p,K}$ be the union of the represented $\Z_p$-square classes whose $p$-adic valuation is smaller than $K$. Then $U_{p,K}\cap \Z$ is a union of residue classes modulo $p^{K+2}$, and the density of $U_{p,K}\cap \Z$ coincides with the Haar measure $\delta_{p,K}(f)$ (say) of $U_{p,K}$. 

It now follows from the Chinese Remainder Theorem that the set of $n$ that are locally represented by $f$ and not divisible by $p^K$ for any $p \in S$ has density 
\[ \delta_{K}(f):= \prod_{p \in S} \delta_{p,K}(f). \]
In particular, the set of locally represented integers has lower density bounded below by $\delta_{K}(f)$, for any $K$. Now letting $K\to\infty$, and noting that $$\delta_{p,K}(f) \nearrow \delta_p(f) \qquad\text{for each fixed $p \in S$,}$$ we deduce that the set of locally represented integers has lower density at least $\prod_{p \in S} \delta_p(f)$. In the other direction, the upper density of locally represented integers is at most
\[ \delta_K(f) + \sum_{p \in S} \frac{1}{p^K} \le \prod_{p \in S} \delta_p(f) + \sum_{p \in S} \frac{1}{p^K}. \] 
Letting $K\to\infty$ bounds this upper density from above by $\prod_{p \in S} \delta_p(f)$. This completes the proof of \eqref{eq:densityhasse2} (and also \eqref{PRODUCTFORMULA}).


\section{Proof of Near Regularity}

\subsection{$n = 1$} For $a \in \Z^{\bullet}$, let $f = at^2$.  It follows from the Global Square Theorem
\cite[Thm. VI.3.7]{Lam} that $D_f = D_{f,\loc}$.

\subsection{$n=2$, anisotropic case}\label{sec:anisotropicbinary} Let $f = ax^2 + bxy + cy^2$ be an integral binary quadratic form, with discriminant $ac-\frac{b^2}{4}$ and Discriminant $\Delta = b^2-4ac$.  Then $f$ is isotropic if and only if $\Delta = B^2$ for some $B \in \N$.
\\ \\
First we suppose that $f$ is anisotropic.  Bernays \cite{Bernays12} showed that there
is a $\kappa_{\Delta} > 0$ such that
\[ \#D_f(X) \sim \frac{\kappa_{\Delta} X}{\log^{1/2} X}. \]
That $\#\mathcal{E}_f(X) = o(D_f(X))$ is also essentially due to Bernays.  A more explicit and modern treatment
is given in work of Odoni
\cite[\S 5]{Odoni77}.
\\ \\
It is much easier to show that $\delta(f) = \delta_{\loc}(f) = 0$.  The set $\mathcal{S}$ of prime numbers $p > 2$ such that
$\left( \frac{-\disc(f)}{p} \right) = -1$ has density $\frac{1}{2}$ within the primes.  For each such prime,
$f_{/\Z_p}$ is anisotropic and represents precisely the two unit square classes, so $\delta_p(f) = 1- \frac{1}{p+1}$.  Therefore
\[ 0 \leq \delta \leq \delta_{\loc}(f) \leq \prod_{p \in \mathcal{S}} \left(1-\frac{1}{p+1}\right) = 0.\]

\subsection{$n = 2$, isotropic case}
Suppose $f$ is isotropic, so $\Delta = B^2$ for some $B \in \Z^+$.   We may assume without loss of generality that $f$ is primitive.   The form $f$ is maximal if and only if $\Delta = 1$, in which case $f$ is a maximal lattice in the hyperbolic
plane, and by \cite[Lemma 29.8]{Shimura} we have $f \cong_{\Z} xy$ and $f$ is universal.  Thus we may
assume that $\Delta > 1$.  Gauss showed \cite[Art. 206]{Gauss} that $f$ is $\SL_2(\Z)$-equivalent to a form
\begin{equation}
\label{GAUSSFORMEQ}
 Ax^2 + Bxy, \ 1 \leq A < B,  \ \gcd(A,B) = 1.
\end{equation}

\begin{lemma}
\label{PAULISOLEMMA1}
Let $a$ and $k$ be coprime positive integers.  Then almost all positive integers have a prime factor
$p \equiv a \pmod{k}$ -- i.e., the exceptional set has density zero in $\Z^+$.
\end{lemma}
\begin{proof}
For each fixed $B \geq 1$, the upper density of the exceptional set is bounded above by the density
of the set of $n \in \Z^+$ that are divisible by no prime $p \equiv a \pmod{k}$ with $p \leq B$, i.e., by
\[ \prod_{p \leq B, \ p \equiv a \pmod{k}} \left(1-\frac{1}{p} \right) . \]
This quantity is at most
\[ \exp \left( - \sum_{p \leq B, \ p \equiv a \pmod{k}} \frac{1}{p} \right). \]
The usual proof of Dirichlet's theorem on primes in arithmetic progressions gives \[\sum_{p \equiv a \pmod{k}} \frac{1}{p} = \infty. \] Taking $B \ra \infty$, the result follows.
\end{proof}

\begin{lemma}
\label{PAULISOLEMMA2}
As $K$ appoaches infinity, the density of the set of integers that are \emph{not} $K$th powerfree
approaches $0$.
\end{lemma}
\begin{proof}
We show the statement for the positive integers, which is sufficient.  If $n \in \Z^+$ is not $K$th power free then there is an integer $m \geq 2$ such that $m^K \mid n$.  It follows that for each
$x > 0$, the number of such $n \leq x$ that are not $K$th power free is at most $\sum_{m =2}^{\infty} \frac{x}{m^K}$, and thus the
upper density of the set of non-$K$th-powerfree positive integers is at most $\sum_{m=2}^{\infty} \frac{1}{m^K}$.  This quantity tends to zero as $K$ approaches infinity, e.g. by the Integral Test.
\end{proof}
\noindent
Let $f = Ax^2 + Bxy$ with $1 \leq A < B$ and $\gcd(x,y) = 1$.  We will prove that almost every integer
locally represented by $f$ is globally represented.  Taking $x = 1$ we see that $D_f$ contains every element
congruent to $A$ modulo $B$, so $\#D_f(X) \gg X$.  So an equivalent statement
of the result is that $\#\mathcal{E}_f(X) = o(X)$.  Fix $K \in \Z^+$, and let $\mathcal{E}_{f,K}(X)$ be the subset of
$\mathcal{E}_f(X)$ consisting of integers that are $K$th power free.  By Lemma \ref{PAULISOLEMMA2},
it is enough to show that for all $K \geq 2$ we have $\#\mathcal{E}_{f,K}(X) = o_K(X)$.  Put
\[ M \coloneqq \prod_{p \mid B} p^{K+v_p(B)}. \]
Let $n \in \mathcal{E}_{f,K}(X)$.  It follows from Lemma \ref{PAULISOLEMMA1} that for a density $1$ subset of $n \in \Z$ we have
that for all $c$ coprime to $M$ there is a prime $p \mid n$ such that $p \equiv c \pmod{M}$.  Henceforth
we restrict to $n$ in this subset and such that $f$ locally represents $n$.  For every prime $p \mid n$
there are $x_{0,p},y_{0,p} \in \Z_p$ such that
\begin{equation}
\label{PISOEQ1}
x_{0,p}(A x_{0,p} + B y_{0,p}) = n.
\end{equation}

\vskip 0.1in

\noindent Step 1: We show that there is $x \in \Z$ such that
\begin{equation}
\label{PISOEQ2}
x \mid n \text{ and } \forall p \mid B, \ x \equiv x_{0,p} \pmod{p^{K+v_p(B)}}.
\end{equation}
To see this, let $r \in \Z$ be such that $r \equiv x_{0,p} \pmod{p^{K+v_p(B)}}$ for all $p \mid B$.  Let $p \mid B$.  It follows from (\ref{PISOEQ1}) that
\[ v_p(x_{0,p}) \leq v_p(n) \leq K \]
and thus
\[ v_p(r) = v_p(x_{0,p}). \]
It follows that there is an integer $u$ coprime to $M$ such that $u\prod_{p \mid B} p^{v_p(x_{0,p})} \equiv r \pmod{M}$.  Because of our assumption on $n$, we may take $u$ to be the class in $\Z/M\Z$ of a prime $q \mid n$, and thus
\[ x \coloneqq q \prod_{p \mid B} p^{v_p(x_{0,p})} \]
satisfies (\ref{PISOEQ2}).
\vskip 0.1in

\noindent Step 2: Let $y_{1,p} \in \Z$ be congruent to $y_{0,p}$ modulo $p^{K+v_p(B)}$.  Reducing (\ref{PISOEQ1})
modulo $p^{K+v_p(B)}$, we get
\[ x(Ax+B y_{1,p}) \equiv n \pmod{p^{K+v_p(B)}}. \]
Since $v_p(x) \leq v_p(n) \leq K$, we get
\[ Ax \equiv Ax + B y_{1,p} \equiv \frac{n}{x} \pmod{p^{v_p(B)}}. \]
Since this holds for all $p \mid B$ we have
\[ Ax \equiv \frac{n}{x} \pmod{B}. \]
Thus we can write $\frac{n}{x} = Ax + By$ for some $y \in \Z$, so $n = x(Ax+By) = f(x,y)$.


\begin{example}
Let $f(x,y) = x^2 + 5xy$.  If $\ell \equiv 4 \pmod{5}$ is a prime number, then $f$ $\Z_p$-represents $\ell$ for all primes $p$: 
for $p \neq 5$, there is a representation with $x = 1$; over $\Z_5$ there is a representation with $y = 0$.  However, it is 
easy to check that the equation $x(x+5y) = \ell$ has no solutions $(x,y) \in \Z^2$.  Thus $f$ is not almost regular, and Dirichlet's theorem on primes in arithmetic progressions gives $\mathcal{E}_f(X) \gg \frac{X}{\log X}$.  
\end{example}

\subsection{$n = 3$, positive case}

Let $f_{/\Z}$ be a positive $r$-ary quadratic form.  The $n$th coefficient of the Eisenstein projection of the
theta series of $f$ can be written as
\[
a_{E}(n) = \prod_{p \leq \infty} \beta_{p}(n)
\]
where for finite $p$,
\[
\beta_{p}(n) = \lim_{k \to \infty}
\frac{\# \{ x \in (\Z/p^{k} \Z)^{r} : f(x) \equiv n \pmod{p^{k}} \}}{p^{(r-1)k}},
\]
and if we write $f(t) = t^T A t$, then
\[
\beta_{\infty}(n) = \frac{r \pi^{r/2} n^{(r-2)/2}}{2 \Gamma\left(\frac{r}{2} + 1\right) \det(A)^{1/2}}.
\]

For a prime $p$, we write the local Jordan splitting of $f$ (as in
\cite{Hanke04}) as
\[
  f(x) = \sum_{j} p^{v_{j}} f_{j}(x_{j}).
\]
If $p > 2$, each $f_{j}$ is one-dimensional, while $f_{j}$ can be
two-dimensional if $p = 2$. We let $r_{p^{k},f}(m) = \# \{ x \in
(\Z/p^{k} \Z)^{r} : f(x) \equiv m \pmod{p^{k}} \}$. A solution
to this congruence is called \emph{good type} if $p^{v_{j}}
x_{j} \not\equiv 0 \pmod{p}$ for some $j$. It is called
\emph{zero type} if $x \equiv 0 \pmod{p}$, and \emph{bad type}
otherwise. In \cite{Hanke04}, Hanke gives a recursive method to determine the number of
good type, bad type, and zero type solutions to $f(x) \equiv m \pmod{p^{k}}$.

\begin{prop}
\label{PROP3.4}
Let $f_{/\Z}$ be a positive ternary quadratic form.  As $X \to \infty$ we have $\mathcal{E}_f(X) = O(\sqrt{X})$. 
%
\end{prop}

\begin{proof}
Let $\theta_{f}(z) \coloneqq \sum_{n=0}^{\infty} r_{f}(n) q^{n}, \quad q = e^{2 \pi i z}$ be the theta series of $f$. We may decompose
\[
  \theta_{f}(z) = E(z) + H(z) + C(z) = \sum_{n=0}^{\infty} a_{E}(n) q^{n} + \sum_{n=0}^{\infty} a_{H}(n) q^{n} + \sum_{n=0}^{\infty} a_{C}(n) q^{n}.
\]
Here $E(z)$ is a weight $3/2$ Cohen-Eisenstein series, $H(z)$ is a weight $3/2$ cusp form that is a linear combination of unary theta series (modular forms
of the shape $\sum_{n \in \Z} \psi(n) n q^{n^{2}}$ where $\psi$ is an odd Dirichlet character),
and $C(z)$ is the projection of $\theta_{f}(z)$ onto the orthogonal complement in the space of weight $3/2$ cusp forms of unary theta series. Let $N(f)$
be the level of $f$.

{\bf Cohen-Eisenstein series}:  It follows from a multivariate form of Hensel's lemma that $\beta_{p}(n) > 0$ if and only if there is a solution to $f(x) = n$ in $\Z_{p}$.

If $p \nmid N(f)$ and $p \nmid n$, Hanke gives the formula
\[
  \beta_{p}(n) = 1 + \frac{1}{p} \legen{-\det(A) n}{p}.
\]
(See Table 1 on the top of page 363 of \cite{Hanke04}.)

Now we consider the case that $p \nmid N(f)$ but $p \mid n$. In this
case, $f_{/\F_p}$ is isotropic (because if $f$
is anisotropic, then $p^{2} \mid N(f)$). As a consequence, $f(x,y,z) = 0$
defines a conic over $\F_{p}$ which has a point on it. Therefore
$f(x,y,z) = 0$ is isomorphic to $\P^{1}$ over $\F_{p}$ and there are
$p+1$ (projective) points on it. This yields $(p+1)(p-1) = p^{2} - 1$
solutions to the equation $f(x,y,z) \equiv 0 \pmod{p}$ with $(x,y,z) \not\equiv (0,0,0) \pmod{p}$. These are ``good type'' solutions in the
terminology of Hanke, and from the recursive formula for good type
solutions (Lemma 3.2 of \cite{Hanke04}) it follows that $\beta_{p}(n)
\geq \frac{p^{2}-1}{p^{2}}$.

If $p \mid N(f)$ but $f_{/\Q_p}$ is isotropic, then
there are finitely many possibilities for
$\beta_{p}(n)$ if $v_{p}(n) < v_{p}(N(f))$.
If $v_{p}(n) \geq v_{p}(N(f))$, then $n$ is \emph{$p$-stable}
in the terminology of Hanke (see Definition 3.6 and Remark 3.6.1),
and it follows that the contribution to $\beta_{p}(p^{2v} n)$ of the
good type and bad type solutions is constant for $v \geq 1$.
Moreover, by Lemma 3.8, this contribution is nonzero, and it follows
that there is an absolute lower bound on $\beta_{p}(n)$ over all
$n$ that are locally represented by $f$.

Finally, if $f_{/\Q_p}$ is anisotropic, then again
there are finitely many possibilities for $\beta_{p}(n)$ if
$v_{p}(n) \leq v_{p}(N(f)) + 1$. The net result is that there is an
absolute constant $C_{E}$ so that
\[
a_{E}(n) \geq C_{E} \beta_{\infty}(n) \prod_{\substack{p \nmid N \\ p \mid n}}
\bigg{(}1 - \frac{1}{p^{2}}\bigg{)} \prod_{\substack{p \nmid N \\ p \nmid n}}
\left(1 + \frac{1}{p} \legen{-\det(A) n}{p}\right),
\]
provided $n$ is locally represented by $f$ and $v_{p}(n) \leq
v_{p}(N(f)) + 1$ for all anisotropic primes $p$.

We have that $\beta_{\infty}(n) = \frac{2 \pi \sqrt{n}}{\det(A)^{1/2}}$.
If we let $\chi_{n}$ be the unique primitive Dirichlet character
with $\chi_{n}(p) = \legen{-\det(A) n}{p}$ for $p$ prime with
$p \nmid \det(A) n$, then we obtain
\[
a_{E}(n) \geq \frac{12 C_{E} \sqrt{n}}{\pi} L(1,\chi_{n})
\prod_{p \mid nN} \left(1 + \frac{\chi_{n}(p)}{p}\right)^{-1}.
\]
Now, we use the ineffective lower bound $L(1,\chi_{n}) \gg
n^{-\epsilon}$. Also note that if $k$ is squarefree, then $\prod_{p | k}
\left(1 - \frac{1}{p}\right) = \frac{\phi(k)}{k} \gg k^{-\epsilon}$
for all $\epsilon > 0$. It follows from this that
\[
a_{E}(n) \gg n^{1/2 - \epsilon}
\]
provided $n$ is locally represented and $v_{p}(n) \leq v_{p}(N(f)) + 1$
for all anisotropic primes $p$.

{\bf Unary theta series}: Because $H(z)$ is a linear combination of
unary theta series, there are finitely many squarefree integers
$b_{1}, b_{2}, \ldots, b_{k}$ so that $a_{H}(n) = 0$ unless $n/b_{i}$
is the square of an integer.

{\bf Cusp form part}: Theorem~1 of \cite{Blomer04} gives that for all $\epsilon > 0$
\[
  |a_{C}(n)| \ll_{f,\epsilon} n^{13/28 + \epsilon}.
\]

Since
\[
r_{f}(n) = a_{E}(n) + a_{H}(n) + a_{C}(n),
\]
the bounds above imply there are only finitely many $n$ for which
\begin{enumerate}
\item[(i)] $a_{H}(n) = 0$,
\item[(ii)] $n$ is locally represented, and
\item[(iii)] $v_{p}(n) \leq v_{p}(N(f)) + 1$ for all anisotropic primes $p$
\end{enumerate}
which are not represented by $f$.

Let $c_{1}, c_{2}, \ldots, c_{\ell}$ be the list of all positive
integers for which $a_{H}(n) = 0$, that are not represented by $f$,
which are locally represented, and for which $v_{p}(c_{i}) \leq
v_{p}(N(f)) + 1$ for all anisotropic primes $p$.

If $n$ is a positive integer that is locally represented by $f$ but
not represented, one possibility is that $a_{H}(n) \ne 0$. In this
case $n/b_{i}$ is a square for some $i$ and there are $O(\sqrt{x})$ many
such $n \leq x$.

If $n$ is a positive integer that is locally represented by $f$ but
not represented and $a_{H}(n) = 0$, let \[S \coloneqq \{ p : p | n \text{ is
  anisotropic and } v_{p}(n) > v_{p}(N(f)) + 1 \} = \{ p_{1},
p_{2}, \ldots, p_{s} \}. \] Corollary 3.8.2 of \cite{Hanke04} implies that
if $n' = n/\prod_{i=1}^{s} p_{i}^{2v_{i}}$ is such that $v_{p}(n')
= v_{p}(N(f))$ or $v_{p}(n') = v_{p}(N(f)) + 1$ for all $p
\in S$, then $r_{f}(n) = 0 \implies r_{f}(n') = 0$.  This $n'$ must be
one of the $c_{i}$ mentioned above, and so $n/c_{i} = \prod_{i=1}^{s}
p_{i}^{2v_{i}}$ is a perfect square. The number of such $n \leq x$ is
$O((\log x)^{t})$, where $t$ is the number of anisotropic primes for
$f$.
\end{proof}

\begin{example}
\label{ROUSEEX}
The form $f(x,y,z) = 3x^{2} + 4y^{2} + 9z^{2}$ locally represents integers $n$ with $v_{2}(n) \ne 1$
that are not in the $\Q_{3}$ square class of $-3$. We will show that the locally represented integers
that are not represented are those of the form $k^{2}$ with all prime factors of $k \equiv 1 \pmod{3}$. This fact was observed by Jones and Pall \cite[Table II]{JonesPall} and a proof of this is given by \cite[Lemma 5]{Schulze-Pillot80}. We illustrate the techniques in the proof of Proposition~\ref{PROP3.4} by giving a self-contained proof.

The decomposition of $\theta_{f}(z) = E(z) + H(z) + C(z)$ has $C(z) = 0$ and
\[H(z) = -\frac{1}{2} \sum_{n=-\infty}^{\infty} n \chi_{3}(n) q^{n^{2}} - 2 \sum_{n=-\infty}^{\infty} n \chi_{3}(n) q^{4n^{2}}. \]
Here $\chi_{3}(n)$ is the non-principal Dirichlet character modulo $3$. It follows that if $n$ is not a square and $n$ is locally represented,
then $a_{E}(n) > 0$ and $a_{H}(n) = 0$ and so $n$ is represented by $f$. If $n = k^{2}$ with $k$ even, then
$f(0,k/2,0) = n$, and if $n = k^{2}$ with $k$ a multiple of $3$, then $f(0,0,k/3) = n$.

Suppose that $n$ is a square which is coprime to $6$. Then the $n$th coefficient of $H(z)$ is $-n^{1/2}$.
A straightforward computation with Yang's formulas for local densities \cite{Yang} shows that if $n$ is a square coprime to $6$, then
\[
  \beta_{p}(n) = \begin{cases}
    \frac{\pi \sqrt{3} n^{1/2}}{9} & \text{ if } p = \infty,\\
    1 + \frac{1}{p} + (\chi_{3}(p)-1) \cdot \frac{1}{p^{v_{p}(n)/2 +1}} & \text{ if } 3 < p < \infty,\\
    2 & \text{ if } p = 3,\\
    1 & \text{ if } p = 2.
\end{cases}
\]
This implies that
\begin{align*}
  a_{E}(n) &= \prod_{p \leq \infty} \beta_{p}(n) = \frac{2 \pi \sqrt{3} n^{1/2}}{9} \prod_{\substack{p > 3}} \left(1 + \frac{1}{p} + \frac{\chi_{3}(p) - 1}{p^{v_{p}(n)/2+1}}\right)\\
  &= \frac{2 \pi \sqrt{3} n^{1/2}}{9} \prod_{p \mid n} \left(1 + \frac{1}{p} + \frac{\chi_{3}(p) - 1}{p^{v_{p}(n)/2+1}}\right) \prod_{\substack{p > 3 \\ p \nmid n}} \left(1 + \frac{\chi_{3}(p)}{p}\right)\\
  &= \frac{\pi \sqrt{3} n^{1/2}}{3} \prod_{p} \frac{p^{2}-1}{p^{2}} \prod_{p \mid n} \frac{p^{2}}{p^{2} - 1} \left(1 + \frac{1}{p} + \frac{\chi_{3}(p) - 1}{p^{v_{p}(n)/2+1}}\right) \prod_{\substack{p > 3 \\ p \nmid n}} \left(1 - \frac{\chi_{3}(p)}{p}\right)^{-1}\\
  &= \frac{\pi \sqrt{3}}{2} n^{1/2} \prod_{p} \frac{p^{2}-1}{p^{2}} \prod_{p} \left(1 - \frac{\chi_{3}(p)}{p}\right)^{-1}
  \prod_{p \mid n} \frac{p^{2}}{p^{2} - 1} \left(1 + \frac{1}{p} + \frac{\chi_{3}(p) - 1}{p^{v_{p}(n)/2+1}}\right) \left(1 - \frac{\chi_{3}(p)}{p}\right)\\ 
  &= \frac{\pi \sqrt{3} n^{1/2}}{2} \frac{1}{\zeta(2)} L(1,\chi_{3}) \prod_{p \mid n} \frac{p^{2}}{p^{2} - 1} \left(1 + \frac{1}{p} + \frac{\chi_{3}(p) - 1}{p^{v_{p}(n)/2+1}}\right) \left(1 - \frac{\chi_{3}(p)}{p}\right)\\ 
  &= \frac{\pi \sqrt{3} n^{1/2}}{2} \cdot \frac{6}{\pi^{2}} \cdot \frac{\pi \sqrt{3}}{9} \prod_{p \mid n} \frac{p^{2}}{p^{2} - 1} \left(1 + \frac{1}{p} + \frac{\chi_{3}(p) - 1}{p^{v_{p}(n)/2+1}}\right) \left(1 - \frac{\chi_{3}(p)}{p}\right)\\
  &= n^{1/2} \prod_{p \mid n} \frac{p^{2}}{p^{2} - 1} \left(1 + \frac{1}{p} + \frac{\chi_{3}(p) - 1}{p^{v_{p}(n)/2+1}}\right) \left(1 - \frac{\chi_{3}(p)}{p}\right).\\
\end{align*}
We have that
\[
  \frac{p^{2}}{p^{2} - 1} \left(1 + \frac{1}{p} + \frac{\chi_{3}(p) - 1}{p^{v_{p}(n)/2+1}}\right) \left(1 - \frac{\chi_{3}(p)}{p}\right) \geq 1
\]
with equality if and only if $\chi_{3}(p) = 1$. It follows that
\begin{align*}
  & n = k^{2} \text{ is not represented by } f \iff a_{E}(n) + a_{H}(n) = 0\\ 
  & \iff a_{E}(n) = n^{1/2} \iff \text{ for all primes } p \mid k, \chi_{3}(p) = 1.
\end{align*}

Thus we have
\[\mathcal{E}_f(X) = 
\# \{ n \leq x : n = k^{2}, \text{ and } p | k \implies p \equiv 1 \pmod{3} \}
\sim \frac{C x^{1/2}}{\log^{1/4}(x)}
\]
for some $C > 0$.  So the upper bound of Proposition \ref{PROP3.4} is sharp up to log factors.
\end{example}

\subsection{$n = 4$, positive case} Let $f_{/\Z}$ be a positive quaternary quadratic form.  Then every
sufficiently large $n \in \Z^+$ that is locally primitively represented -- for all primes $p$
there is $(x_1,x_2,x_3,x_4) \in \Z_p^4$ with at at least one $x_i$ not divisible by $p$ such that
$f(x_1,x_2,x_3,x_4) = n$ -- is primitively represented -- there is $(x_1,x_2,x_3x_4) \in \Z$
with $\gcd(x_1,x_2,x_3,x_4) = 1$ such that $f(x) = n$ \cite[Thm. 11.1.6]{Cassels}. \\ \indent
From this it follows easily that $\#\mathcal{E}_f(X) = O(\sqrt{X})$.  Indeed, every representation of a squarefree
integer is primitive, so $\mathcal{E}_f$ contains only finitely many squarefree exceptions, and there are $O(\sqrt{X})$ integers lying in the same rational square class as one of these squarefree exceptions.  Thus after removing $O(\sqrt{X})$ integers, we may assume 
that $1 \leq n \leq X$ is such that for every integer $m$ lying in the rational square class of $n$, if $m$ is primitively locally 
represented then it is represented.  We claim that $n$ is represented.  It certainly is if it is primitively locally represented.  If not, there is a prime $p \mid n$ such that $f$ $\Z_p$-represents $\frac{n}{p^2}$.
For all primes $\ell \neq p$, we have $\frac{1}{p^2} \in \Z_{\ell}^{\times 2}$ so $f$ $\Z_{\ell}$-represents
$\frac{n}{p^2}$. Thus $f$ locally represents $\frac{n}{p^2}$.   If the representation
is not locally primitive we can repeat the above argument, eventually getting that $f$ primitively represents
$m = \frac{n}{d^2}$ and thus also that $f$ represents $n$.

\subsection{$n \geq 5$, positive case} If $f_{/\Z}$ is a positive $n$-ary quadratic form with $n \geq 5$,
then by Corollary \ref{TARTCOR}, the form $f$ is almost regular, so $\#\mathcal{E}_f(X) = O(1)$.

\subsection{$n \geq 3$, indefinite case}
In this case the result follows from work of Eichler, Kneser, Weil and Hsia, as recorded in \cite[Thm. 5]{Schulze-Pillot13}.  Indeed, if $f_{/\Z}$ is indefinite $n$-ary with $n \geq 4$ then $f$ is regular and thus
$\#\mathcal{E}_f(X) = 0$.  If $f_{/\Z}$ is indefinite ternary then the elements of $\mathcal{E}_f$ lie in finitely
many rational square classes and thus $\#\mathcal{E}_f(X) = O(\sqrt{X})$.

\section{Lattice globalization}
\noindent
An $n$-ary quadratic form $f_{/\Z}$ has \textbf{signature} $(r,s)$ if \[f_{/\R} \cong x_1^2 + \ldots + x_r^2 -
x_{r+1}^2 - \ldots - x_{r+s}^2. \]
The following result amounts to telling us that for all $n \geq 3$, we can find an integral $n$-ary quadratic form that has prescribed 
signature, prescribed $\Z_p$-equivalence class at each of any finite set $S$ of prime numbers, and has $\delta_p(f) = 1$ 
for all primes $p \notin S$.  The underlying ``globalization'' construction is a well known one in the literature -- see e.g. \cite[Theorems 31.6 and 31.7]{Shimura} for a related result -- but as we have not found a result exactly of the form we need, we supply a complete proof.

\begin{thm}
\label{GLOBALIZATION}
Let $n,r,s \in \N$ with $n \geq 3$ and $r+s = n$.   Let $S$ be a finite set of prime numbers, and let $T$ be a subset of $S$.  If $n = 3$, we suppose that
$\# T$ is odd if $rs = 0$ and even otherwise.  For $p \in S$, let $f_p$ be an $n$-ary
quadratic form defined over $\Z_p$ that is anisotropic if and only if $p \in T$.  Then there is an $n$-ary quadratic form $f_{/\Z}$ such that: 
\begin{itemize}
\item For all primes $p \in S$, we have $f_{/\Z_p} \cong f_p$.  
\item The form $f_{/\R}$ has signature $(r,s)$. 
\item If $n = 3$, then for all primes $p \notin S$, we have that $f_{/\Z_p}$ is isotropic and maximal, hence universal.
\item If $n \geq 4$, then for all primes $p \notin S$, we have that $f_{/\Z_p}$ is universal.
\end{itemize}
\end{thm}
\begin{proof}
Let $(f_{\infty})_{/\R}$ be a quadratic form of signature $(r,s)$. 

\vskip 0.1in
\noindent Step 1a): Suppose $n = 3$.  By weak approximation \cite[Thm. 8]{Artin}, there is $d \in \Z^+$ such that for all $p \in S$, we have $d \equiv \disc f_p \pmod{\Q_p^2}$ and $d > 0$ if and only if $s$ is even.  For $p$ a place of $\Q$ (i.e., a prime number or $\infty$) and $g_{/\Q_p}$ a ternary quadratic form, let $\epsilon_p(g) \in \{ \pm 1\}$ be the
Hasse invariant and let $c_p(g) = \epsilon(g) (\frac{-1,-\disc g}{\Q_p}) \in \{ \pm 1\}$ be the Witt invariant.  Recall that for a ternary
quadratic form $g$ over $\Q_p$, we have $c_p(g) = -1$ if and only if $g$ is anisotropic.  For $p$ a place of $\Q$, put
\begin{equation}
\label{GLOBALIZATIONEQ0}
 c_p = \begin{cases} c_p(f_p) & \text{if $p \in S$}, \\
c_p(f_{\infty}) & \text{if $p = \infty$}, \\
1 & \text{otherwise}. \end{cases}
\end{equation}
By virtue of the parity condition imposed on $T$, we have $\prod_{p \leq \infty} c_p = 1$.  Now for $p$ a place of $\Q$, put
\[ \epsilon_p \coloneqq c_p \left(\frac{-1,-d}{\Q_p} \right). \]
By (\ref{GLOBALIZATIONEQ0}) and the properties of the Hilbert symbol, we have that $\epsilon_p = 1$ except for finitely many
$p$ and $\prod_{p \leq \infty} \epsilon_p = 1$.  Since two ternary quadratic forms over $\Q_p$ with the same discriminant square class and Witt invariants are isometric, by \cite[Prop. 7]{CA}, there is a quadratic form $q_{/\Q}$ of signature $(r,s)$ such that
$\disc q \equiv d \pmod{\Q^2}$, $q_{/\Q_p} \cong q_p$ for all $p \in S$ and $q$ is isotropic at all prime numbers $p \notin S$. 

\vskip 0.1in
\noindent Step 1b): Suppose $n \geq 4$.  Fix a finite odd prime $p_0 \notin S$.
For $p$ a place of $\Q$, put
\begin{equation}
\label{GLOBALIZATIONEQ1}
\epsilon_p = \begin{cases}
\epsilon_p(f_p) & \text{if $p \in S \cup \{\infty\}$,}  \\ 1 & \text{if $p \notin S \cup \{p_0,\infty\}$},  \end{cases}
\end{equation}
and take $\epsilon_{p_0} \in \{ \pm 1\}$ to be so that
\[ \prod_{p \leq \infty} \epsilon_p = 1. \]
Now let $(q_{p_0})_{/\Q_{p_0}}$ be an isotropic $n$-ary quadratic form with $\epsilon_{p_0}(q_{p_0}) = \epsilon_{p_0}$.
(Since
\[ \epsilon_{p_0}(t_1^2 + \ldots + t_n^2) = 1\]
and
\[ \epsilon_{p_0}(t_1^2 + \ldots + t_{n-2}^2 + r t_{n-1}^2 + p_0 t_n^2) = -1, \]
this is certainly possible.)  By weak approximation, there is $d \in \Z^+$ such that for all $p \in S \cup \{p_0\}$, we have
$d \equiv \disc f_p \pmod{\Q_p^2}$ and $d > 0$ if and only if $s$ is even.  Applying \cite[Prop. 7]{CA} there is an $n$-ary quadratic form $q_{/\Q}$ of signature $(r,s)$
such that $q_{/\Q_p} \cong q_p$ for all $p \in S \cup \{p_0\}$. 

\vskip 0.1in
\noindent Step 2: We view the quadratic form $q_{/\Q}$ constructed in Step 1 as a quadratic space $V_{/\Q}$.  For each $p \in S$, we may
view $f_p$ as a $\Z_p$-lattice $L_p$ in the quadratic space $V \otimes \Q_p$.  Let $M$ be any maximal $\Z$-lattice in $V$.  By \cite[Thm. 9.4]{Gerstein},
there is a lattice $L$ in $V$ such that for a prime number $p$, we have
\[ L \otimes \Z_p = \begin{cases} L_p &\text{if $p \in S$}, \\ M_p & \text{otherwise}. \end{cases} \]
The quadratic form $f_{/\Z}$ corresponding to $L$ satisfies the desired properties: If $n = 3$, then for all
$p \notin S$, the form $f_{/\Z_p}$ is iostropic and maximal, hence universal, while if $n \geq 4$ then for all $p \notin S$ the
form $f_{/\Z_p}$ is maximal, hence ADC, hence universal.
\end{proof}



\section{Proofs of the remaining theorems}

\subsection{Proof of Theorem \ref{THM1.8}b)} Let $f_{/\Z}$ be a primitive isotropic binary quadratic form
of Discriminant $\Delta$.  By (\ref{PRODUCTFORMULA}) we have $\delta(f) = \prod_p \delta_p(f)$.
For $p > 2$, Case 1 of \S 2.2.1 and (\ref{LOCALDENEQ}) computes $\delta_p(f)$, while
Proposition \ref{PROP2.5} computes $\delta_2(f)$.

\begin{remark}
Our local analysis also shows that for a primitive isotropic binary quadratic form $f_{/\Z}$, the set $D_{f,\loc}$ is a union of congruence classes modulo $\Delta$, so $\delta(f) = \frac{N(\Delta)}{\Delta}$, where $N(\Delta)$ is
the number of congruence classes modulo $\Delta$ that are locally represented by $f$.
\end{remark}

\subsection{Proof of Theorem \ref{THM1.9}} Let $n \geq 3$ and let $f_{/\Z}$ be an $n$-ary quadratic form.  We have
\[ \delta(f) = \prod_p \delta_p(f). \]
Each $\delta_p(f)$ is positive.  Moreover, since $n \geq 3$, for each odd prime $p \nmid \disc(f)$ we have
that $f_{/\Z_p}$ is universal, so $\delta_p(f) = 1$.  It follows that $\delta(f)  > 0$.

\subsection{Proof of Theorem \ref{THM1.10}} Let $f_{/\Z}$ be an anisotropic ternary form.  By Hasse-Minkowski,
$f$ is anisotropic at a nonempty, finite set of places of even cardinality hence over $\Z_p$ for at least one
prime $p$.  As seen in \S 2, $f$ then does not represent any element in the $\Q_p$-square class of $-\disc f$,
which forces $\delta_p(f) < 1$ and thus $\delta(f) = \prod_p \delta_p(f) < 1$.

\subsection{Proof of Theorem \ref{THM1.12}} Let $n \geq 3$, and let $f_{/\Z}$ be an $n$-ary quadratic form.
If $n = 3$ we assume that $f$ isotropic.  For every prime $p$, the form $f_{/\Q_p}$ is universal.  Thus
$f$ is locally universal if and only if it is locally ADC.  Moreover, since every $\Z_p$-square class has positive measure,
these conditions are also equivalent to $\delta_p(f) = 1$ for all $p$ and thus, since $\delta(f) = \prod_p \delta_p(f)$, to $\delta(f) = 1$.  This proves Theorem \ref{THM1.12}a).  Part b) follows from part a)
and Theorem \ref{UNTHM1}: if $f$ is maximal, it is locally maximal and thus locally ADC.

\subsection{Proof of Theorem \ref{THM1.13}} Let $n \geq 1$ and let $f_{/\Z}$ be an $n$-ary
quadratic form. By (\ref{LOCALDENEQ}) we have $\delta_p \in (0,1] \cap \Q$ for all primes $p$. If $n\ge 3$ or $n=2$ but $f_{/\Z}$ is isotropic, then $\delta_p(f)=1$ for all but finitely many primes $p$; hence, $\delta(f) = \prod_p \delta_p(f)$ is rational. The remaining case is that $n \leq 2$ and $f$ is anisotropic, and in this case we've already seen that $\delta(f) = 0$.

\subsection{Proof of Theorem \ref{THM1.14}}

\begin{lemma}
\label{DDLEMMA}
Let $\{a_n\}_{n=1}^{\infty}$ be a sequence with values in $(0,1]$.  We suppose: 
\begin{enumerate}
	\item[(i)] $a_n \ra 0$,
	\item[(ii)] $\sum_{n=1}^{\infty} a_n = \infty$. 
\end{enumerate}
Then: for all $0 \leq \alpha < \beta \leq 1$, there are positive integers $N,n_1,\ldots,n_N$ such that
$\prod_{n=1}^N (1-a_{n_i}) \in (\alpha,\beta)$.
\end{lemma}
\begin{proof}
Put $r \coloneqq \frac{\alpha}{\beta}$.  Choose $K \in \Z^+$ such that for all $n > K$ we have
$1-a_n > r$.  We claim that there is $L \in \Z^+$ such that $\prod_{i=1}^L (1-a_{K+i}) \in (\alpha,\beta)$.  Indeed, since $\sum_{n=1}^{\infty} a_n = \infty$ we have $\prod_{n=1}^{\infty} (1-a_n) = 0$, so certainly $\prod_{i=1}^L
(1-a_{K+i}) < \beta$ for all sufficiently large $L$.  Let $L$ be the smallest such positive integer.  Then, since
$\prod_{i=1}^{L-1} (1-a_{K+i}) > \beta$, we have
\[ \prod_{i=1}^L (1-a_{K+i}) > \beta (1-a_{k+L}) > r\beta = \alpha. \qedhere \]
\end{proof}
\noindent
Now we prove Theorem \ref{DDTHM}.   For $n \in \Z^+$, let $p_n$ be the $n$th odd prime number, and put $a_n \coloneqq \frac{1}{2p_n + 2}$.
This sequence satisfies (i) and (ii) of Lemma \ref{DDLEMMA}.  Thus for all $0 \leq \alpha < \beta \leq 1$
there are $N,n_1,\ldots,n_N$ such that $\prod_{i=1}^N (1-\frac{1}{2p_{n_i}+2}) \in (\alpha,\beta)$.
\\ \indent
Let us first treat the case in which $n \geq 4$.  Then we put $S \coloneqq \{p_{n_1},\ldots,p_{n_N}\}$.  By
the local analysis of \S 2, for every prime $p > 2$ there is an (anisotropic) quadratic form $(f_p)_{/\Z_p}$ with
$\delta_p(f_p) = 1 - \frac{1}{2p+2}$ that represents at least one of $1$ and $r$, hence is primitive.
Applying Theorem \ref{GLOBALIZATION} we get an $n$-ary form $f_{/\Z}$ of signature $(r,s)$
such that $f_{/\Z_p} \cong f_p$ for all $p \in S$ and such that $f_{/\Z_p}$ is universal -- hence primitive,
for all $p \notin S$.  This form $f$ is locally primitive -- hence primitive -- and has
\[ \delta(f) = \prod_{p \in S} \delta_p(f) = \prod_{p \in S} \left(1-\frac{1}{2p+2}\right) \in (\alpha,\beta). \]
Now suppose $n = 3$.  We want to apply the above argument with $S = T$, but there is now a parity requirement
on $T$ that may not be satisfied.  If it is not, we choose $p_{n_{N+1}}$ large enough so that \[\prod_{i=1}^{N+1} \bigg{(}1-\frac{1}{2p_{n_i}+2}\bigg{)} \in (\alpha,\beta)\] and put $S = T = \{p_{n_1},\ldots,p_{n_{N+1}}\}$.  The remainder of the argument is the same.

\subsection{Proof of Theorem \ref{THM1.15}}

Put
\[ \Delta_2 \coloneqq \left\{\frac{1}{2}, \ \frac{5}{6}, \ \frac{11}{12},1 \right\} , \]
and for an odd prime $p$ put
\[ \Delta_p \coloneqq \left\{\frac{1}{2}, \ \frac{p+2}{2p+2}, \ \frac{2p+1}{2p+2},1 \right\}. \]
By Theorem \ref{PRODFORMTHM} and the material of $\S$2, we have $\delta(f) = \prod_p \delta_p(f)$ with $\delta_p(f) \in \Delta_p$ for all $p$
and $\delta_p = 1$ for all but finitely many $p$. 

\vskip 0.1in
\noindent Step 1: If $f$ is isotropic and ADC then it is universal and $\delta(f) = 1$.  So we may assume that $f$ is
anisotropic and thus is anisotropic at $k \geq 1$ primes $p$.  For each of these primes, we have
that $v_2(\delta_p(f)) < 0$ and thus $v_2(\delta(f)) \leq -k$.  It follows that the proportion of attained
values of $\delta$ within $\Q \cap [0,1]$ approaches $0$ as $k \ra \infty$ so it suffices to deal with a
fixed number $k$ of anisotropic primes.

\vskip 0.1in
\noindent Step 2: For each fixed $k \in \Z^+$, let $\mathcal{D}_k$ be the set of rational numbers $\delta = \prod_p \delta_p$ of the above form such that $\delta_p \neq 1$ for at most $k$ primes.  We claim that $\mathcal{D}_k$
is a compact subset of $[0,1]$.  First we observe that every sequence in $\mathcal{D}_1$ has a subsequence
that is either constant, convergent to $1$ or convergent to $\frac{1}{2}$, so $\mathcal{D}_1$ is a sequentially
compact metric space, thus compact.  Next we observe that $\mathcal{D}_k$ is the image of
$\prod_{i=1}^k \mathcal{D}_1$ under the continuous map $(x_1,\ldots,x_k) \mapsto x_1 \cdots x_k$ so is
compact. 

\vskip 0.1in
\noindent Step 3: We show that any compact subset $K \subset \Q \cap [0,1]$ contains only $0\%$ of $\Q \cap [0,1]$
in the sense of height.  We do so by contraposition: For each $n \in \Z^+$,
let $K_n$ be the subset of $K$ consisting of rational numbers with reduced deminator $n$.  Since there are
$\gg N^2$ elements $x \in [0,1] \cap \Q$ with numerator and denominator at most $N$, if $K$ does not contain $0\%$ of
the rational numbers, there is $c > 0$ and infinitely many $N \in \Z^+$ such that the number of elements of $K$ with denominator
at most $N$ is at least $c N^2$, and thus there is an infinite subset $J \subset \Z^+$ such that
\[ \forall n \in J, \ \# K_n > c n: \]
for if not, then the number of rationals in $K$ with reduced denominator at most $N$ would be
at most $\sum_{n=1}^N c n + O(1)$, which is smaller than $c N^2$ for large $N$.  Now we thicken
each $K_n$ to a set $U_n$ by surrounding each point of $K_n$ with the intersection with $[0,1]$ of an open interval of radius $\frac{1}{2n}$.  Since distinct elements of $K_n$ are at least $\frac{1}{n}$ apart, these
intervals are pairwise disjoint, and thus for each $n \in J$ the measure of $U_n$ is at least
$\# K_n \cdot \frac{1}{2n} \geq \frac{c}{2}$.  Let $U$ be the set of reals that belong to $U_n$ for
infinitely many $n$.  Then the measure of $U$ is at least $\frac{c}{2}$ \cite[p. 40]{Halmos}.  Every
element of $U$ is the limit of a sequence in $K$, but $K$ is compact, so $K \supset U$ and thus
$K$ has positive measure. So $K$ cannot be contained in $\Q$.

\subsection{Remarks on $2$-adic properties of the density} We do not know whether the ``locally ADC'' condition in Theorem \ref{THM1.15} can be removed.  One the one hand, we have not completed the classification of densities of quadratic forms over $\Z_2$.  Using our present methods this would be tedious to do without computer assistance.   But this is not really the crux of the matter: consider for instance the problem of finding all densities of
ternary integral quadratic forms $f$ such that $f_{/\Z_2}$ is universal.  In this case, we know everything that we need to know on the quadratic forms side, but we cannot solve the elementary number-theoretic problem of determining which rational numbers arise as $\prod_p \delta_p$ for the known possible values of $\delta_p$.  The proof of Theorem \ref{THM1.15} does not adapt to this case, because it exploits $2$-adic properties of $\delta_p(f)$
and $\delta(f)$.   There is a curious phenomenon here: loosely speaking, ``$v_2(\delta(f))$ wants to be negative but does not need to be negative.''  We make this precise in the following results.

\begin{thm}
\label{NEGATIVEV2THM}
Let $f_{/\Z}$ be a positive definite ternary quadratic form.  Let $U$ be the set of primes $p \equiv 3 \pmod{4}$ such that
$f_{/\Z_p}$ is isotropic.  Suppose that $f_{/\Z_p}$ is ADC for $p = 2$ and for all $p \in U$.  Then
\[ v_2(\delta(f)) < 0. \]
\end{thm}
\begin{proof}
As above, let $S$ be the set of prime numbers $p$ such that $f_{/\Z_p}$ is not both anisotropic and maximal, and let $T$ be the
subset of $S$ of anisotropic places; $\#T$ is odd and thus positive.  Then
\[ \delta(D_f) = \prod_{p \in S} \delta_p(f). \]
We claim that $v_2(\delta_p(f)) \leq 0$ for all $p \in S$ and $v_2(\delta_p(f)) < 0$ for all $p \in T$: certainly, this suffices. 
\begin{itemize}
\item We have $v_2(\delta_2(f)) \leq 0$ by Theorem \ref{LOCALADCTHM}; henceforth we suppose $p > 2$. 
\item Suppose $p \in S \setminus T$, so $f_{/\Z_p}$ is isotropic.  By hypothesis, if $p \equiv 3 \pmod{4}$, then $f_{/\Z_p}$
is ADC, so again Theorem \ref{LOCALADCTHM} gives $v_2(\delta_p(f)) \leq 0$.  Now suppose $p \equiv 1 \pmod{4}$.
Let $X$ be a set of representatives of the $\Z_p$-square classes that $f$ does not represent.  Thus
\[ \delta_p(f) = 1 - \sum_{x \in X} \frac{p-1}{2p^{v(x)+1}}. \]
Since $p \equiv 1 \pmod{4}$, each term has positive $2$-adic valuation, and thus $\delta_p(f) = 0$.  
\item Suppose $p \in T$, so $f_{/\Z_p}$ is anisotropic.  Let $g_{/\Z_p}$ correspond to a maximal lattice containing $f$.  Then arguing as in the previous case we get
\[ \delta_p(f) = \delta_p(g) - \sum_{x \in X} \frac{p-1}{2p^{v(x)+1}}. \]
Since each term of the sum has non-negative $2$-adic valuation and $v_2(\delta_p(g)) < 0$, we get $v_2(\delta_p(f)) = v_2(\delta_p(g)) < 0$. \qedhere
\end{itemize}
\end{proof}


\begin{prop}
\label{DISCONJPROP}
For every $k \in \N$ there is a positive definite ternary quadratic form $f_{/\Z}$ such that $v_2(\delta(f)) \geq k$.
\end{prop}
\begin{proof}
By Dirichlet's Theorem, there is a prime number $p$ such that $p \equiv -1 \pmod{2^{k+2}}$.  Let $S = \{2,p\}$,  $T = \{2\}$ and
\[ f_2 = x^2 + y^2 + z^2, \]
\[ f_p = x^2 + py^2 - pz^2. \]
The form $(f_2)_{/\Z_2}$ is anisotropic, and the form $(f_p)_{/\Z_p}$ is isotropic.  Thus Theorem \ref{GLOBALIZATION} applies, and there is a positive definite ternary $f_{/\Z}$ such that $f_{/\Z_2} \cong f_2$,
$f_{\Z_p} \cong f_p$ and $f_{/\Z_{\ell}}$ is isotropic and maximal for all $\ell \neq 2, p$.  So
\[ \delta(f) = \delta_2(f) \delta_p(f) = \frac{5}{6} \cdot \frac{p+1}{2p}. \]
By our choice of $p$, we have $2^{k+2} \mid p+1$, and the result follows.
\end{proof}


\begin{thebibliography}{BMO11}


\bibitem[Al33]{Albert33} A.A. Albert, \emph{The integers represented by sets of ternary quadratic forms.} Amer. J. Math.
55 (1933), 274--292.

\bibitem[Ar]{Artin} E. Artin, \emph{Algebraic numbers and algebraic functions.} Gordon and Breach Science Publishers, New York-London-Paris 1967.


\bibitem[Be12]{Bernays12} P. Bernays, \emph{\"Uber die Darstellung von positiven, ganzen Zahlen durch die primitiven, bin\"aren quadratischen Formen einer nicht-quadratischen Diskriminante}, Dissertation, G\"{o}ttingen, 1912, available at
\url{http://www.math.uni-bielefeld.de/~rehmann/DML/}

\bibitem[Bl04]{Blomer04} V. Blomer, \emph{Uniform bounds for Fourier coefficients of theta-series with arithmetic applications}. Acta Arith. 114 (2004), 1--21.

\bibitem[BMO11]{BMO11} D. Brink, P. Moree and R. Osburn, \emph{Principal forms $X^2 + nY^2$ representing many integers}. Abh. Math. Semin. Univ. Hambg. 81 (2011), 129--139.

\bibitem[BO08]{Bochnak-Oh08} J. Bochnak and B.-K. Oh, \emph{Almost regular quaternary quadratic forms.} Ann. Inst. Fourier (Grenoble) 58 (2008), 1499--1549.

\bibitem[Ca]{Cassels} J.W.S. Cassels, \emph{Rational quadratic forms.} London Mathematical Society Monographs, 13. Academic Press, Inc., London-New York, 1978.

\bibitem[Cl12]{ADCI} P.L. Clark, \emph{Euclidean quadratic forms and ADC forms: I}  Acta Arith. 154 (2012),  137--159.

\bibitem[CJ14]{ADCII} P.L. Clark and W.C. Jagy, \emph{Euclidean quadratic forms and ADC forms II: integral forms.} Acta Arith. 164 (2014), 265--308.

\bibitem[CO04]{Chan-Oh04} W.K. Chan and B.-K. Oh, \emph{Positive ternary quadratic forms with finitely many exceptions.}
Proc. Amer. Math. Soc. 132 (2004), 1567--1573.


\bibitem[CS91]{Conway-Sloane91} J. Conway and N. Sloane, \emph{Lattices with few distances.} J. Number Theory 39 (1991), 75--90.


\bibitem[DW17]{Doyle-Williams17} G. Doyle and K.S. Williams, \emph{A positive-definite ternary quadratic form does not represent all positive integers}. Integers 17 (2017), Paper No. A41, 19 pages.

\bibitem[Fo96]{Fomenko96} O.M. Fomenko, \emph{Distribution of values of Fourier coefficients of modular forms of weight 1.}
Zap. Nauchn. Sem. S.-Peterburg. Otdel. Mat. Inst. Steklov. (POMI) 226 (1996), Anal. Teor. Chisel i Teor. Funktsiĭ. 13, 196--227, 240; translation in.
J. Math. Sci. (New York) 89 (1998), no. 1, 1050--1071

\bibitem[Ga]{Gauss} C.F. Gauss, \emph{Disquisitiones arithmeticae.} Translated and with a preface by Arthur A. Clarke. Revised by William C. Waterhouse, Cornelius Greither and A. W. Grootendorst and with a preface by Waterhouse. Springer-Verlag, New York, 1986.

\bibitem[Ge]{Gerstein} L.J. Gerstein, \emph{Basic quadratic forms.}
Graduate Studies in Mathematics, 90. American Mathematical Society, Providence, RI, 2008.


\bibitem[H]{Halmos} P.R. Halmos, \emph{Measure Theory}. D. Van Nostrand Company, Inc., New York, N. Y., 1950.

\bibitem[Ha04]{Hanke04} J. Hanke, \emph{Local densities and explicit bounds for representability by a quadratric form.} Duke Math. J. 124 (2004),  351--388.

\bibitem[JKS97]{JKS97} W.C. Jagy, I. Kaplansky and A. Schiemann, \emph{There are 913 regular ternary forms.} Mathematika 44 (1997),  332--341.

\bibitem[JP39]{JonesPall} B. W. Jones, G. Pall, \emph{Regular and semi-regular positive ternary quadratic forms}, Acta Math. 70 (1939), 165--191.

\bibitem[L]{Lam} T.-Y. Lam, \emph{Introduction to quadratic forms over fields.} Graduate Studies in Mathematics, 67. American Mathematical Society, Providence, RI, 2005.

\bibitem[La08]{Landau08} E. Landau, \emph{\"Uber die Einteilung der positiven ganzen Zahlen in vier Klassen nach der Mindestzahl der zu ihrer additiven Zusammensetzung erforderlichen Quadrate}. Arch. Math. Phys.
13 (1908), 305--312.

\bibitem[LO14]{LemkeOliver14} R.J. Lemke Oliver, \emph{Representation by ternary quadratic forms.} Bull. Lond. Math. Soc. 46 (2014),  1237–-1247.

\bibitem[MO]{MO} \url{https://mathoverflow.net/questions/75399}

\bibitem[MO06]{Moree-Osburn06} P. Moree and R. Osburn, \emph{Two-dimensional lattices with few distances.}
Enseign. Math. (2) 52 (2006),  361--380.

\bibitem[Od77]{Odoni77}  R.W.K. Odoni, \emph{A new equidistribution property of norms of ideals
in given classes}. Acta Arithmetica 33 (1977), 53--63.

\bibitem[RP46]{Ross-Pall46} A.E. Ross and G. Pall, \emph{An extension of a problem of Kloosterman.}
Amer. J. Math. 68 (1946), 59--65.

\bibitem[Se]{CA} J.-P. Serre, \emph{A course in arithmetic}. Translated from the French. Graduate Texts in Mathematics, No. 7. Springer-Verlag, New York-Heidelberg, 1973.

\bibitem[Sh]{Shimura} G. Shimura, \emph{Arithmetic of quadratic forms}. Springer Monographs in Mathematics. Springer, New York, 2010.

\bibitem[SP80]{Schulze-Pillot80} R. Schulze-Pillot, \emph{Darstellung durch Spinorgeschlechter tern\"{a}rer quadratischer Formen}. J. Number Theory 12 (1980), no. 4, 529--540.

\bibitem[SP13]{Schulze-Pillot13} R. Schulze-Pillot, \emph{Representation of quadratic forms by integral quadratic forms.} Quadratic and higher degree forms, 233--253,
Dev. Math., 31, Springer, New York, 2013.

\bibitem[Wa]{Watson} G. L. Watson, \emph{Integral quadratic forms.} Cambridge Tracts in Mathematics and Mathematical Physics, No. 51. Cambridge University Press, New York, 1960.

\bibitem[Ya98]{Yang} T. Yang, \emph{An explicit formula for local densities of quadratic forms}. J. Number Theory, 72 (1998), 309--356.


\end{thebibliography}
\end{document}